\documentclass[11pt]{amsart}
\usepackage{latexsym,amssymb,amsmath}
\usepackage{amsmath,amsfonts}
\usepackage{graphicx}
\usepackage{verbatim}
\usepackage{enumitem}

\newtheorem{theorem}{Theorem}[section]
\newtheorem{lemma}[theorem]{Lemma}
\newtheorem{proposition}[theorem]{Proposition}

\newtheorem{definition}[theorem]{Definition}
\newtheorem{remark}[theorem]{Remark}
\allowdisplaybreaks

\usepackage{hyperref}
\hypersetup{
  colorlinks   = true,    
  linkcolor    = blue,    
  citecolor    = red     }

\newcommand{\R}{\mathbb{R}}

\newcommand{\lb}{\langle}
\newcommand{\rb}{\rangle}
\newcommand{\I}{\operatorname{I}}
\newcommand{\II}{\operatorname{II}}
\newcommand{\III}{\operatorname{III}}
\renewcommand{\d}{\operatorname{d}\!}

\renewcommand{\Re}{\operatorname{Re}}
\renewcommand{\Im}{\operatorname{Im}}
\renewcommand{\hat}{\,\widehat}
\newcommand{\wt}{\widetilde}
\newcommand{\sgn}{\operatorname{sgn}}

\def\eqnn{\begin{eqnarray*}}
\def\eeqnn{\end{eqnarray*}}
\def\eqn{\begin{eqnarray}}
\def\eeqn{\end{eqnarray}}
\newcommand{\nc}{\newcommand}
\nc{\be}{\begin{equation}}
\nc{\ee}{\end{equation}}
\nc{\ba}{\begin{eqnarray}}
\nc{\ea}{\end{eqnarray}}
\nc{\eps}{\epsilon}
\def\prf{\begin{proof}}
\def\endprf{\end{proof}}

\begin{document}

\title[KGS on $\R^+$]{Low-regularity global well-posedness for the Klein-Gordon-Schr\"odinger system on $\R^+$}

\author{E. Compaan}
\address{Department of Mathematics, Massachusetts Institute of Technology} 
\email{compaan@mit.edu}

\author{N. Tzirakis}
\address{Department of Mathematics, University of Illinois at Urbana-Champaign}
\email{tzirakis@illinois.edu}

\thanks{The first author was supported by a National Physical Science Consortium fellowship and by NSF MSPRF \#1704865. The second author's work was supported by a grant from the Simons Foundation (\#355523 Nikolaos Tzirakis) and by Illinois Research Board, RB18051.}

\date{}

\begin{abstract} In this paper we establish an almost optimal well--posedness and regularity theory for the Klein-Gordon-Schr\"odinger system on the half line. In particular we prove local-in-time well--posedness for rough initial data in Sobolev spaces of negative indices. Our results are consistent with the sharp well--posedness results that exist in the full line case and in this sense appear to be sharp. Finally we prove a global well--posedness result by combining the $L^2$ conservation law of the Schr\"odinger part with a careful iteration of the rough wave part in lower order Sobolev norms in the spirit of the work in \cite{CHT}.
\end{abstract}

\maketitle

\section{Introduction}
We consider the Klein-Gordon-Schr\"odinger (KGS) system \eqref{eq:KGS} where the Schr\"odinger part $u$ is a complex-valued function and the Klein-Gordon part $n$ is a real valued function. This classical model describes the interaction of a nucleon field with a neutral meson field. We are especially interested in describing the local and global-in-time dynamics of the problem on the half line. The boundary value problem that arises is significantly harder to analyze than its real line counterpart and as far as we know the sharp local and/or global well--posedness theory is unknown. More precisely we consider the following system of dispersive partial differential equations (PDE)
\begin{equation} \label{eq:KGS}
 \begin{cases}
  iu_t + \Delta u = -nu, \quad (x,t) \in \R^+ \times \R^+, \\
  n_{tt} + (1 - \Delta)n = |u|^2, \\
  u(x,0) = u_0(x) \in H^{s_0}(\R^+) \\
  n(x,0) = n_0(x) \in H^{s_1}(\R^+), \quad n_t(x,0) = n_1(x) \in {H}^{s_1 - 1}(\R^+), \\
  u(0,t) = g(t) \in H^{\frac{2s_0 + 1}4}(\R^+), \quad n(0,t) = h(t) \in H^{s_1}(\R^+),
 \end{cases}
\end{equation}
with the additional compatibility conditions $u_{0}(0)=g(0)$ when $s_0>\frac12$, $h(0)=n_0(0) $ when $s_1>\frac12$. The compatibility conditions are necessary since the solutions we are interested in are    continuous space-time functions for $s>\frac12$. 

Because of the prescribed boundary data, there are no conserved quantities for the system in its full generality. On the other hand as we show in Section \ref{globalsec}, the $L^2$ norm of the Schr\"odinger part is conserved, whenever the Schr\"odinger boundary data $g$ is identically zero. In this manuscript we use no other conservation law and thus our results are valid if one replaces $nu$ and $|u|^2$ by $-nu$ and/or $-|u|^2$ respectively. 

The KGS on $\Bbb R^n$ is extensively studied and sharp results have been obtained by many authors, see \cite{GTV}, \cite{CHT}, \cite{Pecher2}, and \cite{Pecher} and the references therein. The existence and uniqueness of local-in-time solutions is proved either by the use of Strichartz estimates for Sch\"odinger and Klein-Gordon equations or with the implementation of the restricted norm method of Bourgain (\cite{bou} and \cite{bou1}). Because the nonlinearities in the system are bilinear, the restricted norm method usually provides sharp results at least for local solutions. The lowest regularity level for well--defined strong solutions can be found in \cite{Pecher} at least for the low dimensional cases. In particular in 2d, the author proves that the KGS is locally well--posed for 
$u_0\in H^{s_0},\ n_{0}\in H^{s_1},\ n_1 \in H^{s_1-1}$ where $s_0>-\frac{1}{4}$, $s_1 \geq -\frac12$ and $s_1-2s_0<\frac32$, $s_1-2<s_0<s_1+1$. Moreover he shows that the problem is ill-posed if either $s_0<-\frac14$ or $s_1<-\frac12$. The same numerology holds in the one dimensional case as one can easily check by adapting the estimates in \cite{Pecher} to 1d. In this paper we also obtain the local theory up to the endpoint $(s_0,s_1)=(-\frac14,-\frac12)$ and in this sense our result is optimal. Before we proceed any further we should define the notion of the solution in the presence of a boundary.

For initial and boundary value problems the local well--posedness is given by the following definition. The Fourier restriction spaces $X^{s_0,b}$ and $Y^{s_1,b}$ will be defined below.
\begin{definition} 
We say \eqref{eq:KGS} is locally wellposed in $H^{s_0}(\R^+)\times H^{s_1}(\R^+)$, if for any 
$u_0\in H^{s_0}(\R^+)$,    
$n_0\in H^{s_1}({  \R}^+)$, 
$n_1\in H^{s_1-1}({ \R}^+)$,
$g\in H^{\frac{2s_0+1}{4}}(\R^+)$, and $h\in H^{s_1}(\R^+)$,
  with the additional compatibility conditions mentioned above, the integral equation \eqref{eq:duhamel} below has a unique solution in
\be\label{def:lwpspace}
\big[X^{s_0,b}  \cap C^0_tH^{s_0}_x  \cap C^0_xH^{\frac{2s_0+1}{4}}_t \big] \times \big[Y^{s_1,b} \cap C^0_tH^{s_1}_x  \cap C^0_xH^{s_1}_t \big],
\ee
on a time interval $[0,T]$, for $b<\frac12$ and $T$ sufficiently small depending only on the norms of the initial and boundary data.     
 Furthermore,   the solution depends continuously on the initial and boundary data. If the time $T$ can be taken arbitrarily large we say that the system is globally well--posed.
\end{definition}

In this paper we establish the regularity properties of the KGS on a half line  using the tools that are available in the case of the real line where the PDE is fully dispersive. To prove our main theorems we rely on the Duhamel formulation of the nonlinear system adapted to the boundary conditions, see \eqref{eq:duhamel} below,  which expresses the nonlinear solution as the superposition of the linear evolutions, which incorporate the boundary and the initial data, with the nonlinearity.  To this end we extend the data into the whole line and use the Laplace and Fourier transform methods to   set up an equivalent integral equation (on $\Bbb R\times \Bbb R$) of the solution.  We analyze the integral equation  using  the restricted norm method and multilinear $L^2$ convolution estimates. At the end, via standard arguments, we prove existence of local-in-time solutions by Banach's fixed point argument. 

More precisely, in what follows we first establish a sharp local-in-time well--posedness theory and in addition we prove that the nonlinear part of the system is in a much smoother space than the corresponding linear part. Propagation of regularity is then used to establish uniqueness of strong solutions. We should mention here, that the uniqueness of solutions for initial and boundary value problems that are constructed through extensions, is not an easy byproduct of the iteration process. The reason being that we do not know if the fixed point solutions of the Duhamel operators have restrictions on the half line which are independent of the extensions of the data. In the absence of conservation laws or at least global a priori bounds this can become a hard problem to overcome, see \cite{ck}. In this paper, following the method we established in \cite{ct}, we show how one can use the local smoothing estimates to answer the uniqueness question. We thus obtain the following theorem:

\begin{theorem}
 For $s_0 \in (-\frac14, \frac12)$ and $s_1 \in (-\frac12, \frac12)$, the KGS system \eqref{eq:KGS} is locally well-posed on $\R^+$. Moreover, the nonlinear part of the solution exhibits additional smoothness:
 \begin{align*}
 u - W_0^t(u_0, g) &\in C^0_t H^{s_0 + a_0}(\R^+ \times [0,T]), \\
 n - V_0^t(n_0,n_1,h) &\in C^0_t H^{s_1 + a_1} (\R^+ \times [0,T]), 
 \end{align*}
 for any $a_0 < \min\{ \frac12,s_1 + \frac12\}$ and $a_1 < 2s_0 - s_1 + \frac12$. 
\end{theorem}

\begin{remark}
Our result matches the sharp (up to endpoints) result for the KGS system on $\R$, which is locally well-posed in $H^{-\frac14+} \times H^{-\frac12}$. Local theory at this regularity on $\R^2$ and $\R^3$ can be found in \cite{Pecher}, \cite{Pecher2} -- these results imply the corresponding well-posedness on $\R$. Counterexamples demonstrating sharpness in two and three dimensions can be found in \cite{bhht}, \cite{Pecher2} -- the same counterexamples serve in the one-dimensional case.  
\end{remark}

\begin{remark}
 We establish a well--posedness result for regularities $s_0, s_1 <\frac12$. Our methods also give results at higher regularities. However, obtaining such results introduces additional complications in our estimates of the boundary operator. We therefore elect to focus on the low-regularity case in this work. 
\end{remark}

Based on the local theory and the argument in \cite{CHT}, we then prove a new global well--posedness result: 
\begin{theorem}
 For $s_0=0$ and $s_1 \in (-\frac12, \frac12)$, the KGS system \eqref{eq:KGS} with boundary data $h=0$, is globally well--posed on $\R^+$ and the solution satisfies 
 $$\|u(t)\|_{L^2(\Bbb R^{+})}=\|u_0\|_{L^2(\Bbb R^{+})}$$
 and
 $$\|n(t)\|_{H^{s_1}(\Bbb R^{+})}+\|\partial n(t)\|_{H^{s_{1}-1}(\Bbb R^{+})}\leq C \exp \Bigl(c|t|\left( \|u_0\|_{L^2(\Bbb R^+)}+\|h\|_{H^{s_1}(\Bbb R^+)}\right)\Bigr)$$
 where the constants depend on the norms $\|u_{0}\|_{L^2}$, $\|n_{0}\|_{H^{s_1}}$, $\|n_1\|_{H^{s_1-1}}$ and $\|h\|_{H^{s_1}}$.
\end{theorem}

In general terms, the well--posedness theory for initial and boundary value problems can be substantially advanced by considering solutions in the $X^{s,b}$ spaces with $b$ less than, but close to, $\frac12$. We remark that the boundary operators are never in $X^{s,b}$ for any $b>\frac12$, \cite{ET}. But to prove the above theorem we are forced to work with $b=\frac13$. This is a smaller norm and thus in order to iterate the local solutions we need to improve on certain multilinear estimates. Moreover, since the boundary terms produce additional corrections in Duhamel's formula, the iteration process on the half line is more complicated than on $\Bbb R$. To carry out the iteration we use a variant of an idea that appeared in \cite{CHT}. By the conservation of the $L^2$ norm of the Schr\"odinger part we know that from one local step to the other, the norm of the Schr\"odinger solution doesn't grow. To successfully then go from local to global solutions we should control the norms of the Klein-Gordon part. In this paper we present a necessary modification of  this process by showing how one can control efficiently the additional terms that arise due to the boundary operator, using the method of odd extensions. We present the details of the argument in Section \ref{globalsec}.

We now discuss briefly the organization of the paper. In the remainder of this section we standardize the notation we use throughout the paper, we define the Banach spaces of the solutions and write down explicitly the new Duhamel's formula that incorporates the boundary value problem and the extension of the half-line data. We note that the solution is constructed on $\Bbb R$ but its restriction on $\R^+$ satisfies the PDE in an appropriate sense. In Section \ref{estimates} we state the linear and nonlinear estimates that we need in order to establish the local and global well--posedness theory for our problem. In Section \ref{lwp}, assuming the estimates in Section \ref{estimates}, we prove the local theory along with the smoothing estimates that the nonlinear system satisfies. In Section \ref{uni} we show that the solutions we have constructed are unique and independent of the extension of the initial data, using the smoothing estimates of Section \ref{lwp}. Section \ref{globalsec} gives the details of the globalizing technique which can in principle be used in a variety of dispersive systems on the half line. In Section \ref{proofs} we have collected some proofs of the linear and the nonlinear estimates. This is probably the most technical part of the paper. Finally at the end of the paper in Appendix \ref{appendA} we show how one can use the Laplace transform to obtain the solution of the linear Klein-Gordon on the half-line.

\subsection{Notation \& Function Spaces}\label{notation}

We define the one-dimensional Fourier transform by
\[ \widehat{f}(\xi) = \mathcal{F}_x f(\xi) = \int_{\R} e^{-i x \xi} f(x) \d x. \]
The Sobolev space $H^s(\R)$ is defined by the norm
\[ \|f\|_{H^s(\R)} = \| \lb \xi \rb^s \hat{f} \|_{L^2}, \quad \text{ where } \quad \lb \xi \rb = \sqrt{ 1 + |\xi|^2}.\]
 The half-line Sobolev spaces for negative indices are defined as follows.
\begin{definition}
 For $s> -\frac12$, we define $H^s(\R^+)$ by
 \[ H^s(\R^+) = \{ f \in \mathcal{D}(\R^+) \; : \; f = F|_{\R^+} \text{for some } F \in H^s(\R)\} .\]
 The corresponding norm is
 \[\|f\|_{H^s(\R^+)} = \inf\{ \| F\|_{H^s(\R)} \; : \;  F|_{\R^+} = f\} .\]
For $s>-3/2$,  we define $H^{s}(\R^+)= \Bigl(H_0^{-s}(\R^+)\Bigr)^*$ with the usual dual norm.   
\end{definition}

The restriction definition of $H^s(\R^+)$ is confined to $s > -\frac12$ since at lower regularities restriction is no longer defined. 
We will also use the $X^{s,b}$ spaces (\cite{bou}, \cite{bou1}) corresponding to the Schr\"{o}dinger and wave flows. These are defined for functions on the full space $\R_x \times \R_t$ by the norms 
\begin{align*} 
\|u\|_{X^{s,b}} &= \Bigl\| \lb \xi \rb^s \lb \tau - \xi^2 \rb^b \widehat{u}(\xi,\tau) \Bigr\|_{L^2_\xi L^2_\tau}, \\
\|n\|_{Y^{s,b}_\pm} &= \Bigl\| \lb \xi \rb^s \lb \tau \mp \xi \rb^b \widehat{n}(\xi,\tau) \Bigr\|_{L^2_\xi L^2_\tau}, \\
\|n\|_{Y^{s,b}} &= \inf_{n=n_+ + n_-} \left(\|n_+\|_{Y_+^{s,b}}+ \|n_-\|_{Y_-^{s,b}} \right).
\end{align*}

The characteristic function on $[0,\infty)$ will be denoted by $\chi$. We choose $\eta \in C^\infty(\R)$ to be a smooth bump function such that $\eta = 1$ on $[-1,1]$ and $\operatorname{supp} \eta \subset [-2,2]$. We also define a scaled version of $\eta$ by $\eta_T(\cdot) = \eta(\cdot/T)$, so that $\eta_T = 1$ on $[-T,T]$. 

The notation $a \lesssim b$ indicates that $a \leq Cb$ for some absolute constant $C$. The expression $a \gtrsim b$ is defined similarly, and $a \approx b$ means that $a \lesssim b$ and $a \gtrsim b$. The notation $a+$ indicates $a + \epsilon$, where $\epsilon$ can be arbitrarily small. We define $a-$ similarly.

\subsection{Fixed-Point Equation Formulation}

We will denote the linear Schr\"odinger flow on $\R$ by 
\[ e^{i t \Delta} u_0 = \mathcal{F}^{-1}[e^{-it |\cdot|^2} \hat{u_0}(\cdot)].\]

To solve the nonlinear problem \eqref{eq:KGS}, we begin by considering the corresponding linear initial-boundary-value problems. The Schr\"odinger problem 
\begin{equation} \label{eq:SchrIBVP}
 \begin{cases}
  iu_t + \Delta u = 0, \quad x,t \in \R^+, \\
  u(x,0) = u_0(x) \in H^{s_0}(\R^+), \quad u(0,t) = g(t) \in H^{\frac{2s+1}4}(\R^+),
 \end{cases}
\end{equation}
has been studied in \cite{ET}, \cite{bonaetal}. Let $u_0^e \in H^{s_0}(\R)$ be an extension of $u_0$ such that $\|u_0^e\|_{H^{s_0}(\R)} \lesssim \| u_0\|_{H^{s_0}(\R^+)}$. Then the solution to \eqref{eq:SchrIBVP} can be written in the form $W^t_0(u_0^e, g)$, where 
\begin{equation} \label{eq:SchrIBVPformula}
 W^t_0(u_0^e, g) = e^{it\Delta} u_0^e + W^t_0(0, g-p), \quad \text{where} \quad p(t) = \eta(t) [ e^{it\Delta} u_0^e]_{x=0}.
 \end{equation}

Let $D$ be the Fourier multiplier operator defined by 
\[ \hat{Df}(\xi) = \sgn(\xi) \sqrt{1 + \xi^2} \hat{f}(\xi).\]
This somewhat unusual choice of $D$ is convenient for our calculations. 
The linear Klein-Gordon equation on $\R$, which can be written as
\[ \begin{cases} 
  n_{tt} - (iD)^2 n = 0, \quad x,t \in \R \\
  n(x,0) = n_0(x), \quad n_t(x,0) = n_1(x),
   \end{cases} \]
has solution
\[ n(x,t) = W_R^t\bigl(n_0(x),n_1(x)\bigr) = W_1^t\bigl(n_0(x)\bigr) + W_2^t\bigl(n_1(x)\bigr), \]
where $W_1^t$ and $W_2^t$ are spatial Fourier multiplier operators defined by 
\[ \mathcal{F}\Bigl(W_1^t\bigl(n_0\bigr)\Bigr)(\xi) = \Re\bigl(e^{it \sgn(\xi) \sqrt{\xi^2 + 1}}\bigr)\widehat{n_0}(\xi) \qquad \mathcal{F}\Bigl(W_2^t\bigl(n_1\bigr)\Bigr)(\xi) = \frac{\Im e^{it \sgn(\xi) \sqrt{\xi^2 + 1}}}{\sgn(\xi)\sqrt{\xi^2+1}}\widehat{n_1}(\xi). \] 

It will be important to note that this linear flow preserves oddness -- that is, if $n_0$ and $n_1$ are odd, then the solution remains odd. This can be verified by noting that the Fourier transform of an odd real-valued function is odd and purely imaginary. The Fourier multipliers above are even and real-valued, so the transform of the linear flow remains odd and purely imaginary. The inverse Fourier transform of such a function is real and odd. 

We also need the solution to the linear Klein-Gordon initial-boundary value problem with zero initial data:
\begin{equation} \label{eq:KGibvp}
 \begin{cases}
    n_{tt} - (iD)^2 n  = 0, \quad x,t \in \R^+  \\
    n(x,0) = n_t(x,0) = 0, \\
    n(0,t) = h(t). 
   \end{cases} 
\end{equation}
The following solution formula can be derived via Laplace transforms. A proof is given in Appendix \ref{appendA}. 
\begin{lemma}\label{AB}
Suppose $h$ is a Schwarz class function on $\R^+$. Then the solution $V^t_0(0, h)$ to \eqref{eq:KGibvp} can be written as $\frac1{2\pi} (A + B)$, where
\begin{align*}
 A &=  \int_{-1} ^1 e^{ i \mu t - x \sqrt{1 - \mu^2}} \rho(x\sqrt{1 - \mu^2}) \hat{h}(\mu) \d\mu  \\
 B &= \int_{-\infty} ^\infty e^{-it \mu \sqrt{1 + 1/\mu^2} + i\mu x } \hat{h}(-\mu \sqrt{1+ 1/\mu^2}) \frac{1}{\sqrt{1 + 1/\mu^2}} \d \mu. 
\end{align*}
Here we write $\hat{h}$ for $\hat{\chi h}$. 
\end{lemma}

Now let $n_{0}^e$ and $n_{1}^e$ be extensions to $\R$ of $n_0$ and $n_1$ respectively such that 
\[ \| n_{0}^e\|_{H^{s_1}(\R)} \lesssim \|n_0\|_{H^{s_1}(\R^+)} \qquad \|n_{1}^e\|_{H^{s_1 - 1}(\R)} \lesssim \|n_1\|_{H^{s_1 - 1}(\R^+)}.\] 
We note that such extensions certainly exist. One possible choice of extension is the odd extension, For $f \in H^{s}(\R^+)$ with $s \in (-\frac32,\frac12)/\{-\frac12\}$, we define $f^\text{odd} \in H^s(\R)$ by its action on $\varphi \in H^{-s}(\R)$, as follows: 
\[ f^\text{odd}[\varphi] := f\Bigl[\chi \Bigl(\varphi(x) - \varphi(-x)\Bigr) \Bigr]. \]
This makes sense since for $-s \in (-\frac12, \frac32) \backslash \{ \frac12\}$  an odd $H^{-s}(\R)$ function can be restricted to obtain an $H^{-s}_0(\R^+)$ function. Furthermore, $f^\text{odd}$ agrees with $f$ on functions supported on $\R^+$ and $\| f^\text{odd}\|_{H^s(\R)} \lesssim \|f \|_{H^s(\R^+)}$. 

To reduce the Klein-Gordon evolution on $\R$ to a first-order equation in time, we define
\begin{equation} \label{eq:phiDef}
\phi_\pm (x) = n_{0}^e(x) \mp iD^{-1} n_{1}^e(x) \in H^{s_1}(\R).
\end{equation}
Then the solution to the linear Klein-Gordon equation on $\R$ with data $n_{0}^e$ and $n_{1}^e$ is 
\[ \frac12 \Bigl[ e^{itD}\phi_+ + e^{-itD} \phi_- \Bigr]. \] 
We can then express the solution $V_0^t(\phi_\pm, h)$ to the linear Klein-Gordon on the half-line with initial data $(n_0,n_1)$ and boundary data $h$ as
\begin{equation} 
V_0^t(\phi_\pm, h) = \frac12 \Bigl[ e^{itD}\phi_+ + e^{-itD} \phi_- \Bigr] + V_0^t(0,h - r)(x),
\end{equation}
where 
\begin{equation}\label{eq:rDef}
 r(t) = \frac12 \Bigl[ e^{itD}\phi_+ + e^{-itD} \phi_- \Bigr]_{x=0}.
\end{equation}
Then $V^t_0(\phi_\pm, h)$ solves the linear Klein-Gordon on $\R^+_x \times \R^+_t$.

We are now ready to express the solution of \eqref{eq:KGS} as a fixed point of an integral operator. The following formulae hold on $[0,T]$, for $0<T<1$. 
\begin{equation}\label{eq:duhamel} \left\{
\begin{array}{l}
\Gamma_1 u(t) = u(t) = \eta_T(t) W_0^t\big(u_0^e, g  \big) -i \eta_T(t) \int_0^t e^{i(t- t^\prime)\Delta}   F(u,n) \d t'  +i\eta_T(t) W_0^t\big(0,  q  \big), \\
\Gamma_2 n(t) = n(t) = \eta_T(t) V_0^t\big(\phi_{\pm}, h\big) + \frac12\eta(t) (n_+ +  n_-) -\frac12 \eta_T(t) V_0^t(0,z), 
\end{array}\right.
\end{equation}

where
\begin{align} \label{eq:def_duhamel_compnts}
\begin{split}
F(u,n)&=\eta_T(t) n u \qquad \qquad \qquad
q(t)=  \eta_T(t ) \Bigl[\int_0^t e^{i(t-t^\prime)\Delta} F(u,n)\d t' \Bigr]_{x=0} \\
n_\pm &=   \mp i \int_0^t e^{\pm i (t-t^\prime) D} G(u )\d t' \\
G(u )&=\eta_T(t) D^{-1}  |u|^2,  \qquad \qquad
z(t)=  \eta_T(t) [n_++n_-]_{x=0}
\end{split}
\end{align}
and the linear flows $W_0^t$ and $V_0^t$ are defined in \eqref{eq:SchrIBVPformula} and \eqref{eq:phiDef}--\eqref{eq:rDef} respectively.

\subsection{Fundamental Estimates}

We recall the embedding $X^{s,b}, Y^{s,b}\hookrightarrow C^0_t H^{s}_x$ which holds for $b>\frac{1}{2}$, as well as the following inequalities, from \cite{bou1},and \cite{GTV}. 
For any $s,b\in \R$ we have
\begin{equation}\label{eq:xs1}
\|\eta(t) e^{it\Delta} u_0 \|_{X^{s,b}}\lesssim \|u_0\|_{H^s}.
\end{equation}
For any $s\in \R$,  $0\leq b_1<\frac12$, and $0\leq b_2\leq 1-b_1$, we have
\begin{equation}\label{eq:xs2}
\Big\| \eta(t) \int_0^t  e^{i(t-t^\prime)\Delta}  F(t^\prime ) dt^\prime \Big\|_{X^{s,b_2} }\lesssim   \|F\|_{X^{s,-b_1} }.
\end{equation}
Moreover, for $T\leq1$, and $-\frac12<b_1<b_2<\frac12$, we have
\begin{equation}\label{eq:xs3}
\|\eta_T(t) F \|_{X^{s,b_1}}\lesssim T^{b_2-b_1} \|F\|_{X^{s,b_2}}.
\end{equation}
Analogous inequalities hold for the norms $Y^{s,b}_\pm$.

For our global theory, we recall the following inequalities (see, e.g. \cite{CHT}), which hold for $0 \leq b < \frac12$ and $T \lesssim 1$:
\begin{equation}\label{eq:T_power_linear}
\begin{split}
 \| \eta_T(t) e^{it\Delta} u_0 \|_{X^{0,b}} &\lesssim T^{\frac12 - b} \| u_0\|_{L^2} \\
 \| \eta_T(t) e^{\pm itD} \phi \|_{Y^{s_1, b}} &\lesssim T^{\frac12 - b} \Bigl(\| n_0\|_{H^{s_1}} + \|n_1\|_{H^{s_1-1}}\Bigr). 
 \end{split}
\end{equation}

Finally, we need the following lemma regarding multiplication by characteristic functions:
\begin{lemma}{\cite{ET}} \label{char_func_lemma} Assume $f \in H^s(\R^+)$. 
 If $-\frac12 < s < \frac12$, then $\| \chi f \|_{H^s(\R)} \lesssim \| f \|_{H^s(\R^+)}$. If in addition $f(0) = 0$, the same statement holds for $\frac12 < s < \frac32$.
\end{lemma}

\section{A Priori Estimates}\label{estimates}

\subsection{Linear Estimates}

For the linear Schr\"odinger equation, we have the following estimates, which were proved in \cite{ET}. 

\begin{lemma}\label{schr_ivp_cont}
 For any $u_0 \in H^s(\R)$, we have $\eta(t) e^{it\Delta}u_0 \in C_x^0H^{\frac{2s+1}{4}}_t(\R \times \R)$ with the estimate
 \[ \| \eta e^{it\Delta} u_0\|_{L^\infty_xH_t^{\frac{2s+1}4}} \lesssim \| u_0 \|_{H^s(\R)}.\] 
\end{lemma}

\begin{lemma}\label{schr_ibvp_cont}
 For any $g$ such that $\chi g \in H^{\frac{2s+1}4}(\R)$, we have $W^t_0(0,g) \in C^0_tH^s_x(\R \times \R)$, and $\eta(t) W^t_0(0,g) \in C_x^0H_t^\frac{2s+1}4(\R \times \R)$. 
\end{lemma}

\begin{lemma}\label{schr_ibvp_xsb}
 Let $b \leq \frac12$. Then for $g$ such that $\chi g \in H^{\frac{2s+1}4}(\R)$, we have 
 \[ \| \eta(t) W_0^t(0,g)\|_{X^{s,b}} \lesssim \| \chi g\|_{H^\frac{2s+1}4(\R)}.\] 
 Furthermore, for $T \lesssim 1$, we have 
 \[ \| \eta_T(t) W_0^t(0,g)\|_{X^{s,b}} \lesssim  T^{1/2 - |b|} \| \chi g\|_{H^\frac{2s+1}4(\R)}.\] 
\end{lemma}
The proof of the first part of Lemma \ref{schr_ibvp_xsb} above is in \cite{ET}. The second statement comes from the fact that $\widehat{\eta_T}(\tau) = T \,\widehat{\eta}(\tau T)$, via a change of variables argument very similar to what we will use to prove Lemma \ref{ibvp_xsb} below. 

For the Klein-Gordon part, the following estimates hold. Proofs are in Sections \ref{ivp_cont_prf}-\ref{ibvp_cont_prf}. 

\begin{lemma}\label{ivp_cont}
 For $g \in H^s(\R)$, we have $\eta(t) e^{\pm itD} g \in C^0_xH^s_t(\R \times \R)$ with the bound
 \[ \| \eta(t) e^{\pm i t D } g \|_{L^\infty_xH^s_t} \lesssim \|g\|_{H^s(\R)}.\] 
\end{lemma}

\begin{lemma}\label{ibvp_xsb}
 Fix $s, b \in \R$. Then for $h$ such that $\chi h \in H^{s}(\R)$, we have 
 \[ \| \eta(t) V_0^t(0,0,h)\|_{Y^{s,b}} \lesssim \| \chi h \|_{H^{s}(\R)}.\] 
 Furthermore, for $T \lesssim 1$, we have 
 \[ \| \eta_T(t) V_0^t(0,0,h)\|_{Y^{s,b}} \lesssim T^{1/2 - |b|} \| \chi h \|_{H^{s}(\R)}.\] 
\end{lemma}

\begin{lemma}\label{ibvp_cont}
For $h$ such that $\chi h \in H^s(\R)$, we have $V^t_0(0,0,h) \in C_t^0H_x^s(\R \times \R)$ and $ V_0^t(0,0,h) \in C^0_xH^s_t(\R \times \R)$. 
\end{lemma}

\subsection{Nonlinear Estimates}

For the integral term in the Schr\"odinger equation, we have the following estimate. 

\begin{proposition}[\cite{ET1}] \label{schr_duhamel_est}
 For any $b < \frac12$, we have 
\begin{multline*}
 \left\| \eta(t) \int_0^t e^{i(t-t') \Delta} F \d t' \right\|_{C_x^0 H^\frac{2s_0+1}{4}_t(\R \times \R)} \\
 \lesssim \begin{cases} \| F\|_{X^{s_0,-b}} &\text{for } -\frac12 < s_0 \leq \frac12,\\
           \|F\|_{X^{s_0,-b}} + \| \int \lb \lambda + \xi^2 \rb^\frac{2s_0-3}{4} | \hat{F}(\xi, \lambda)| \d \xi \|_{L^2_\xi}  &\text{for } s_0 > \frac12.  
          \end{cases}
 \end{multline*}
\end{proposition}

This result appears for $s \geq 0$ in \cite{ET1}; the proof there applies to $s > -\frac12$ as well. The following proposition is used to control the correction term which appears on the right-hand side in the above estimate.

\begin{proposition}[\cite{ET1}]\label{schr_duhamel_corr} For any $s_0, s_1$ and any $a_0 \geq 0$ such that
$$ 
\frac12-s_0<a_0<\min\left(\frac12, s_1+\frac12,s_1-s_0+1\right),
$$ 
there exists $\epsilon>0$ such that for $\frac12-\epsilon<b <\frac12$, we have
\begin{equation*}
\Big\| \int_{\R } \lb \lambda+\xi^2 \rb^{\frac{2(s_0+a_0)-3}{4}}   | \hat{nu}(\xi,\lambda)| d\xi  \Big\|_{L^2_\xi} \lesssim \|u\|_{X^{s_0,b}}  \|n\|_{Y^{s_1,b}}.
\end{equation*}
\end{proposition}

For the global theory argument, we need the following estimate.
\begin{proposition} \label{T_power_est}
For $T \lesssim 1$ and $0 \leq b < \frac12$, we have 
\begin{equation*}
 \| \eta_T(t) W^t_0(0,q)\|_{X^{0,b}} \lesssim T^{1-2b} \| F\|_{X^{0,-b}}, \quad
\text{where}
\quad 
 q(t) = \Bigl[ \eta_T(t) \int_0^t e^{i(t-t')\Delta} F(t') \d t' \Bigr]_{x=0} .
\end{equation*}
\end{proposition}
This follows from results above together with a change of variables. The proof is in Section \ref{T_power_est_prf}. 

For the integral term in the wave equation, we have the following estimate, which is proved in Section \ref{wave_duhamel_est_prf}.  

\begin{proposition} \label{wave_duhamel_est}
 For any $b < \frac12$, we have 
\begin{multline*}
 \left\| \eta(t) \int_0^t e^{\pm i (t-t') D} G \d t' \right\|_{C^0_xH^{s_1}_t(\R \times \R)} \\
 \lesssim 
 \begin{cases} \| G\|_{Y^{s_1, -b}_\pm} + \left\| \lb \lambda \rb^{s_1} \int_{|\xi| \gg |\lambda|} \lb \lambda \mp \xi \rb^{-1} |\widehat{G}(\xi, \lambda)| \d \xi \right\|_{L^2_\lambda} &\text{for } -\frac12 < s_1 < 0, \\
  \| G \|_{Y^{s_1,-b}_\pm} &\text{for } 0 \leq s_1 \leq \frac12.
  \end{cases}
\end{multline*}
\end{proposition}

To control the correction term for negative $s_1$, we need the estimate below. The proof is in Section \ref{prop:smooth23_prf}. 

\begin{proposition}\label{prop:smooth23}
 For $s_0 > \min\{ -\frac12, \frac12 - 2b\}$ with $-\frac12 < s_1 + a_1 < 0$, we have 
\[  \left\| \lb \lambda \rb ^{s_1 +a_1} \int_{|\xi| \gg |\lambda|} \frac{|\mathcal{F}\bigl(D^{-1}(u\overline{v})\bigr)(\xi ,\lambda)|}{\lb \lambda \mp \xi \rb} \d \xi \right\|_{L^2_\lambda} \lesssim \|u\|_{X^{s_0,b}}\|v\|_{X^{s_0,b}}. \]

\end{proposition}

We finish this section by presenting the proofs of the main nonlinear estimates for the KGS system.
\begin{proposition} \label{KG_bilinear} For any $\frac{s_0}{2}+\frac14<b<\frac12$, any $-\frac{1}{4} < s_0 <\frac12$, $s_1 > -\frac12$, and any $$a<2s_{0}-s_{1}+2b-\frac{1}{2}$$
we have
$$\|u\overline{v}\|_{Y^{s_{1}+a,-b}}\lesssim \|u\|_{X^{s_{0},b}}\|v\|_{X^{s_{0},b}}.$$
\end{proposition}
\begin{proof}
By duality and after renaming the $L^2$ based functions it is enough to consider the estimate
\begin{equation}\label{fnonlinear}
\iiiint\frac{f(\xi_1,\tau_1)g(\xi-\xi_1,\tau-\tau_1)h(\xi,\tau)\langle \xi \rangle^{s_1+a}d\xi_{1}d\tau_{1}d\xi d\tau}{{\langle \xi_1 \rangle}^{s_{0}}{\langle \xi-\xi_{1} \rangle}^{s_{0}}{\langle \tau-\xi \rangle}^{b}{\langle \tau_1+\xi_{1}^2\rangle}^{b}{\langle \tau-\tau_1-(\xi-\xi_1)^2\rangle}^{b}}\lesssim \|f\|_{L^2}\|g\|_{L^2}\|h\|_{L^2}.
\end{equation}
By applying the Cauchy Schwarz inequality first in the $\xi,\tau$ variables and then in the $\xi_1,\tau_1$ variables and integrating the $\tau_1$ integral it is enough to bound the integral
$$\sup_{\xi,\tau}\int\frac{\langle \xi \rangle^{2(s_1+a)}d\xi_{1}}{{\langle \xi_1 \rangle}^{2s_{0}}{\langle \xi-\xi_{1} \rangle}^{2s_{0}}{\langle \tau-\xi \rangle}^{2b}{\langle \tau-\xi^2+2\xi \xi_1\rangle}^{4b-1}}.$$
Since $4b-1<2b$ this reduces to
\begin{equation}\label{red1}
\sup_{\xi}\int\frac{\langle \xi \rangle^{2(s_1+a)}d\xi_{1}}{{\langle \xi_1 \rangle}^{2s_{0}}{\langle \xi-\xi_{1} \rangle}^{2s_{0}}{\langle \xi(\xi-2\xi_1-1) \rangle}^{4b-1}}.
\end{equation}
Notice that for $|\xi_1|\lesssim 1$ the integral becomes
$$\sup_{\xi}{\langle \xi \rangle}^{2s_1-2s_0+2a}{\langle \xi^2 \rangle}^{1-4b}\int_{|\xi_1|\lesssim 1}d\xi_1$$
which is finite as long as 
$$a<s_0-s_1+4b-1.$$ Thus from now on we assume that $|\xi_1|\gg 1$.
We now denote
$$\lambda_1=\tau-\xi$$
$$\lambda_2=\tau_1+\xi_{1}^2$$
$$\lambda_3=\tau-\tau_{1}-(\xi-\xi_{1})^2$$
and we notice that $\lambda_1-\lambda_3-\lambda_2=\xi(\xi-2\xi_1-1)$ while if we change variables $\xi_1\rightarrow \xi_{1}-\frac{1}{2}$ inside the integral the identity becomes
$$\lambda_1-\lambda_3-\lambda_2=\xi(\xi-2\xi_1).$$ 
We first consider the resonant cases R1 and R2.
\\
{\bf R1:} $|\xi-2\xi_1|\lesssim 1$. In this case $\langle \xi \rangle \sim \langle \xi_1 \rangle$ and since $\xi-\xi_1=\xi-2\xi_1-\xi_1$ we also have that
$$\langle \xi-\xi_1\rangle \sim\langle \xi \rangle \sim \langle \xi_1 \rangle.$$
Then we have that 
\begin{multline*}
\eqref{red1}\sim \sup_{\xi}\int\frac{\langle \xi \rangle^{2(s_1+a)}d\xi_{1}}{{\langle \xi_1 \rangle}^{2s_{0}}{\langle \xi-\xi_{1} \rangle}^{2s_{0}}{\langle \xi(\xi-2\xi_1) \rangle}^{4b-1}}\\ \lesssim \sup_{\xi}{\langle \xi \rangle}^{2s_1+2a-4s_0-4b+1}\int_{|\xi-2\xi_1|\lesssim 1}\frac{1}{{|\xi-2\xi_1|}^{4b-1}}d\xi_1\lesssim \sup_{\xi}{\langle \xi \rangle}^{2s_1+2a-4s_0-4b+1}\lesssim 1
\end{multline*}
as long as 
$$a<2s_0-s_1+2b-\frac{1}{2}.$$
Notice that to integrate we used the fact that $4b-1<1$ along with the inequality $$\frac{1}{\langle \xi(\xi-2\xi_1)\rangle} \lesssim \frac{1}{\langle \xi \rangle |\xi-2\xi_1|}$$ which is justified for large $|\xi - 2\xi_1|\lesssim 1$.
\\
{\bf R2:} $|\xi|\lesssim 1$. In this case we first apply the Cauchy Schwarz inequality to the $\xi_1,\tau_1$ variables and then to the $\xi,\tau$ variables and thus it is enough to bound
$$\sup_{\xi_1}\int\frac{\langle \xi \rangle^{2(s_1+a)}d\xi}{{\langle \xi_1 \rangle}^{2s_{0}}{\langle \xi-\xi_{1} \rangle}^{2s_{0}}{\langle \xi(\xi-2\xi_1) \rangle}^{4b-1}}.$$
but this integral is majorized by $\sup_{\xi_1}{\langle \xi_1 \rangle }^{-4s_0-4b+1}\int_{|\xi|\lesssim 1}\frac{1}{|\xi|^{4b-1}}d\xi$ which is finite as long as $b<\frac12$.
\\
\\
We now consider the nonresonant case $|\xi|\gg 1$ and $|\xi-2\xi_1|\gg 1$. In this case we have that $\max_{i=1,2,3}|\lambda_{i}| \gtrsim \langle \xi \rangle \langle \xi-2\xi_1 \rangle $. We consider the cases A and B. Case A is when $\lambda_1$ is the maximum and case B is the case when $\lambda_2$ is the maximum. The third case, where $| \lambda_3|$ is the maximum, is similar to case B and will be omitted.
\\
\\
{\bf Case A:} $|\lambda_1|=\max_{i=1,2,3}|\lambda_{i}| \gtrsim \langle \xi \rangle \langle \xi-2\xi_1 \rangle $. We have three subcases.
\\
\\
Case 1: $|\xi_1|\gg |\xi|$. In this case we have that $|\lambda_1|\gtrsim \langle \xi \rangle \langle \xi_1 \rangle$ and that $\langle \xi_1 \rangle \sim \langle \xi-\xi_1 \rangle$. By applying Cauchy Scwharz as before in estimate \eqref{fnonlinear} it is enough to bound
$$\sup_{\xi,\lambda_1}\langle \lambda_1\rangle^{-2\epsilon}\iint\frac{\langle \xi \rangle^{2s_1+2a-2b+2\epsilon}d\xi_1 d\lambda_2}{{\langle \xi_1 \rangle}^{4s_{0}+2b-2\epsilon}{\langle \lambda_2 \rangle}^{2b}{\langle \lambda_3 \rangle}^{2b}}.$$
First observe that $4s_0+2b-2\epsilon>0$ and thus 
$${\langle \xi \rangle }^{4s_0+2b-2\epsilon} \lesssim {\langle \xi_1 \rangle }^{4s_0+2b-2\epsilon}.$$
We now change variables from $\xi_1$ to $\lambda_3$ (for fixed $\xi,\lambda_1,\lambda_2$) using $\lambda_1-\lambda_3-\lambda_2=\xi^2-2\xi\xi_1$ to obtain $d\lambda_3=2|\xi|d\xi_1$. Thus the integral is majorized by
$$\sup_{\xi,\lambda_1}\langle \xi \rangle^{2s_1+2a-4b-4s_0-1+4\epsilon}\langle \lambda_1\rangle^{-2\epsilon}\iint_{|\lambda_2|,|\lambda_3|\lesssim |\lambda_1|}\frac{d\lambda_2 d\lambda_3}{\langle \lambda_2\rangle^{2b} \langle \lambda_3\rangle^{2b}} \lesssim 1$$
for 
$$a<2s_0-s_1+2b+\frac{1}{2}.$$
\\
\\
Case 2: $|\xi|\gg |\xi_1|$. In this case $|\lambda_1|\gtrsim \langle \xi \rangle^2 $ and $\langle \xi \rangle \sim \langle \xi-\xi_1 \rangle$. We thus have to bound
$$\sup_{\xi,\lambda_1}\langle \lambda_1\rangle^{-2\epsilon}\iint_{|\xi_1|\lesssim |\xi|}\frac{\langle \xi \rangle^{2s_1+2a-2s_0-4b+4\epsilon}d\xi_1 d\lambda_2}{{\langle \xi_1 \rangle}^{2s_{0}}{\langle \lambda_2 \rangle}^{2b}{\langle \lambda_3 \rangle}^{2b}}.$$
In the case that $s_0\leq 0$, we dismiss $\lambda_3$ and use the fact that $\langle \xi_1\rangle^{-2s_0}\lesssim \langle \xi\rangle^{-2s_0}$ and $$\int_{|\lambda_2|\lesssim |\lambda_1|}\frac{d\lambda_2}{\langle \lambda_2 \rangle^{2b}} \lesssim |\lambda_1|^{1-2b}.$$ Thus the integral is finite as long as $b=\frac12-\epsilon$ and $a<2s_0-s_1+2b-\frac12$.
\\
In the case that $0<s_0\leq\frac12$ we integrate $$\int_{|\xi_1| \lesssim |\xi|}\frac{1}{\lb \xi_1\rb^{2s_0}}d\xi_1\lesssim |\xi|^{1-2s_0+\epsilon},$$ we dismiss $\lambda_3$, and integrating in $\lambda_2$ as above we have a finite integral as long as $b=\frac12-\epsilon$ and $a<2s_0-s_1+2b-\frac12$.
\\
Case 3: $|\xi|\sim |\xi_1|$.
\\
Case 3a: $|\xi-2\xi_1| \gtrsim |\xi|$. 

In this subcase $|\lambda_1|\gtrsim \langle\xi\rangle^2\sim \langle\xi_1\rangle^2$ since $\langle\xi\rangle \sim \langle \xi_1 \rangle$. In the case that $s_{0}\leq 0$ we can easily bound
$$\langle\xi-\xi_1\rangle^{-2s_0}\lesssim \langle\xi\rangle^{-2s_0}.$$
Then by Cauchy-Schwarz inequality and dismissing the $\lambda_3$ weight we need to bound
$$\sup_{\xi,\lambda_{1}}\langle \xi \rangle^{2s_1+2a-4s_{0}-4b+4\epsilon}\langle \lambda_1\rangle^{-2\epsilon}\iint_{|\xi_1|\sim |\xi|}\frac{d\xi_1d\lambda_2}{\langle \lambda_2\rangle^{2b}}\lesssim \sup_{\xi,\lambda_{1}}\langle \xi \rangle^{2s_1+2a-4s_{0}-4b+4\epsilon+1}\langle \lambda_1\rangle^{-2\epsilon+1-2b}.$$
This is finite for any $b$ close to $\frac{1}{2}-$ and any $a<2s_0-s_1+2b-\frac12$. In the case that $0<s_0\leq \frac12$ by Cauchy-Schwarz and dismissing $\lambda_3$ we need to bound
\begin{multline*}
\sup_{\xi,\lambda_{1}}\langle \xi \rangle^{2s_1+2a-2b}\langle \lambda_1\rangle^{-2\epsilon}\iint_{|\xi_1|\sim |\xi|}\frac{d\xi_1d\lambda_2}{\langle \xi_1 \rangle^{2b+2s_{0}-4\epsilon}\langle \xi-\xi_1\rangle^{2s_0}\langle \lambda_2\rangle^{2b}}\\ \lesssim \sup_{\xi,\lambda_{1}}\langle \xi \rangle^{2s_1+2a-2b}\langle \lambda_1\rangle^{-2\epsilon+1-2b}\int_{|\xi_1|\sim |\xi|}\frac{d\xi_1}{\langle \xi_1 \rangle^{2b+2s_{0}-4\epsilon}\langle \xi-\xi_1\rangle^{2s_0}}\\ \lesssim \sup_{\xi,\lambda_{1}}\langle \xi \rangle^{2s_1+2a-2b}\langle \lambda_1\rangle^{-2\epsilon+1-2b}\langle \xi \rangle^{1-4s_0-2b+4\epsilon}. 
\end{multline*}
Again, this is finite for $b=\frac12-$ and any $a<2s_0-s_1+2b-\frac12$.
\\
\\
Case 3b: $|\xi-\xi_1| \gtrsim |\xi|$. Notice that in this case $\langle\xi-\xi_1 \rangle \sim \langle\xi \rangle\sim \langle \xi_1 \rangle$. Proceeding like in Case 1 it is enough to bound
$$\sup_{\xi,\lambda_1}\langle \lambda_1\rangle^{-2\epsilon}\iint_{|\xi_1|\sim |\xi|}\frac{\langle \xi \rangle^{2s_1+2a-4s_0-2b+2\epsilon}d\xi_1 d\lambda_2}{{\langle \xi-2\xi_1 \rangle}^{2b-2\epsilon}{\langle \lambda_2 \rangle}^{2b}{\langle \lambda_3 \rangle}^{2b}}.$$
Since $$\int_{|\xi_1|\sim |\xi|}\frac{d\xi_1 }{{\langle \xi-2\xi_1 \rangle}^{2b-2\epsilon}}\lesssim |\xi|^{1-2b+2\epsilon},$$
if we dismiss $\lambda_3$ and integrate in $|\lambda_2|\lesssim |\lambda_1|$ as above the estimate follows for $b=\frac12-\epsilon$ and for any $a<2s_0-s_1+2b-\frac12$.
\vskip 0.3in
\noindent
{\bf Case B:} $|\lambda_2|=\max_{i=1,2,3}|\lambda_{i}| \gtrsim \langle \xi \rangle \langle \xi-2\xi_1 \rangle $. We have three subcases.
\\
\\
Case 1: $|\xi_1|\gg |\xi|$. In this case we have that $|\lambda_2|\gtrsim \langle \xi \rangle \langle \xi_1 \rangle$ and that $\langle \xi_1 \rangle \sim \langle \xi-\xi_1 \rangle$. By applying Cauchy Scwhartz as before in estimate \eqref{fnonlinear} it is enough to bound
\begin{equation} \label{caseB1} \sup_{\xi_1,\lambda_2}\langle \lambda_2\rangle^{-2\epsilon}\iint\frac{\langle \xi \rangle^{2s_1+2a-2b + 2 \epsilon}d\xi d\lambda_1}{{\langle \xi_1 \rangle}^{4s_{0}+2b - 2 \epsilon}{\langle \lambda_1 \rangle}^{2b}{\langle \lambda_3 \rangle}^{2b}}.
\end{equation}
First observe that for $a<b-s_1$ we can dismiss the power in the $\xi$ variable and thus we need to bound
$$\sup_{\xi_1,\lambda_2}\langle \lambda_2\rangle^{-2\epsilon}\iint\frac{d\xi d\lambda_1}{{\langle \xi_1 \rangle}^{4s_{0}+2b}{\langle \lambda_1 \rangle}^{2b}{\langle \lambda_3 \rangle}^{2b}}.$$
We now change variables from $\xi$ to $\lambda_3$ (for fixed $\xi_1,\lambda_1,\lambda_2$) using $\lambda_1-\lambda_3-\lambda_2=\xi^2-2\xi\xi_1$ to obtain $d\lambda_3=2|\xi_1|d\xi$. Thus the integral is majorized by
$$\sup_{\xi,\lambda_2}\langle \xi_1 \rangle^{-2b-4s_0-1+2\epsilon}\langle \lambda_2\rangle^{-2\epsilon}\iint_{|\lambda_1|,|\lambda_3|\lesssim |\lambda_2|}\frac{d\lambda_1 d\lambda_3}{\langle \lambda_1\rangle^{2b} \langle \lambda_3\rangle^{2b}} \lesssim 1$$
which is acceptable.

If $a \geq b - s_1$, we go back to \eqref{caseB1} and instead of dismissing $\lb \xi \rb^{2s_1 + 2a + 2 \epsilon}$, we bound it by $\lb \xi_1 \rb^{2s_1 + 2a + 2 \epsilon}$ and proceed as before. 
\\
\\
Case 2: $|\xi|\gg |\xi_1|$. In this case $|\lambda_2|\gtrsim \langle \xi \rangle^2 $ and $\langle \xi \rangle \sim \langle \xi-\xi_1 \rangle$. We thus have to bound
$$\sup_{\xi_1,\lambda_2}\langle \lambda_2\rangle^{-2\epsilon}\iint_{|\xi_1|\lesssim |\xi|}\frac{\langle \xi \rangle^{2s_1+2a-2s_0-4b+4\epsilon}d\xi d\lambda_1}{{\langle \xi_1 \rangle}^{2s_{0}}{\langle \lambda_1 \rangle}^{2b}{\langle \lambda_3 \rangle}^{2b}}.$$
Now we dismiss $\lambda_3$ and note that since $s_0\leq 0$ we have that $\langle \xi_1\rangle^{-2s_0}\lesssim \langle \xi\rangle^{-2s_0}$.
Then
$$\int_{|\lambda_1|\lesssim |\lambda_2|}\frac{d\lambda_2}{\langle \lambda_1 \rangle^{2b}} \lesssim |\lambda_2|^{1-2b}$$
and we need to bound 
$$\sup_{\xi_1}\int_{|\xi_1|\lesssim |\xi|}\langle \xi \rangle^{2s_1+2a-4s_0-4b+4\epsilon}d\xi.$$
This integral is finite as long as $a<2s_0-s_1+2b-\frac12$.

For $0 < s_0 \leq \frac12$, we go back to \eqref{fnonlinear} and apply Cauchy-Schwarz in $\xi$, $\tau$ and then in $\xi_1, \tau_1$ so that we arrive at an expression containing the supremum over $\xi$. The argument then proceeds as in previous cases. 
\\
\\
Case 3: $|\xi|\sim |\xi_1|$.
\\
Case 3a: $|\xi-2\xi_1| \gtrsim |\xi|$. 
\\
In this subcase $|\lambda_2|\gtrsim \langle\xi\rangle^2$. In the case that $s_{0}\leq 0$ we can easily bound
$$\langle\xi-\xi_1\rangle^{-2s_0}\lesssim \langle\xi\rangle^{-2s_0}.$$
Then by Cauchy-Schwarz inequality and dismissing the $\lambda_3$ weight we need to bound
$$\sup_{\xi_1,\lambda_{2}}\langle \xi_1 \rangle^{2s_1+2a-4s_{0}-4b+4\epsilon}\langle \lambda_2\rangle^{-2\epsilon}\iint_{|\xi_1|\sim |\xi|}\frac{d\xi d\lambda_1}{\langle \lambda_1\rangle^{2b}}\lesssim \sup_{\xi_1,\lambda_{2}}\langle \xi_1 \rangle^{2s_1+2a-4s_{0}-4b+4\epsilon+1}\langle \lambda_2\rangle^{-2\epsilon+1-2b}\lesssim 1$$
for $\frac12 - b> 0 $ suffuciently small and any $a<2s_0-s_1+2b-\frac12$. In the case that $0<s_0\leq \frac12$ by Cauchy-Schwarz and dismissing $\lambda_3$ we need to bound
\begin{multline*}
\sup_{\xi_1,\lambda_{2}}\langle \xi_1 \rangle^{2s_1+2a-2b}\langle \lambda_2\rangle^{-2\epsilon}\iint_{|\xi_1|\sim |\xi|}\frac{d\xi d\lambda_1}{\langle \xi \rangle^{2b+2s_{0}-4\epsilon}\langle \xi-\xi_1\rangle^{2s_0}\langle \lambda_1\rangle^{2b}}\\ 
\lesssim \sup_{\xi_1,\lambda_{2}}\langle \xi_1 \rangle^{2s_1+2a-2b}\langle \lambda_2\rangle^{-2\epsilon+1-2b}\int_{|\xi_1|\sim |\xi|}\frac{d\xi}{\langle \xi \rangle^{2b+2s_{0}-4\epsilon}\langle \xi-\xi_1\rangle^{2s_0}}\\ 
\lesssim \sup_{\xi_1,\lambda_{2}}\langle \xi_1 \rangle^{2s_1+2a-2b}\langle \lambda_2\rangle^{-2\epsilon+1-2b}\langle \xi_1 \rangle^{1-4s_0-2b+4\epsilon}\lesssim 1
\end{multline*}
for $b=\frac12-$ and any $a<2s_0-s_1+2b-\frac12$.
\\
\\
Case 3b: $|\xi-\xi_1| \gtrsim |\xi|$. Notice that in this case $\langle\xi-\xi_1 \rangle \sim \langle\xi \rangle\sim \langle \xi_1 \rangle$. Proceeding like in Case 1 it is enough to bound
$$\sup_{\xi_1,\lambda_2}\langle \lambda_2\rangle^{-2\epsilon}\iint_{|\xi_1|\sim |\xi|}\frac{\langle \xi_1 \rangle^{2s_1+2a-4s_0-2b+2\epsilon}d\xi d\lambda_1}{{\langle \xi-2\xi_1 \rangle}^{2b-2\epsilon}{\langle \lambda_1 \rangle}^{2b}{\langle \lambda_3 \rangle}^{2b}}.$$
Since $$\int_{|\xi_1|\sim |\xi|}\frac{d\xi }{{\langle \xi-2\xi_1 \rangle}^{2b-2\epsilon}}\lesssim |\xi_1|^{1-2b+2\epsilon},$$
if we dismiss $\lambda_3$ and integrate in $|\lambda_1|\lesssim |\lambda_2|$ as above the estimate follows for $b=\frac12-\epsilon$ and for any $a<2s_0-s_1+2b-\frac12$.

\end{proof}

\begin{proposition} \label{KG_bilinearI} For any $\frac13\leq b<\frac12$ but close to $\frac12$, any $-\frac14<s_0\leq 0$, $-\frac12<s_1\leq 0$  and any $$a<s_1+2b-\frac12$$
we have
$$\|uv\|_{X^{s_{0}+a,-b}}\lesssim \|u\|_{X^{s_{0},b}}\|v\|_{Y^{s_{1},b}}.$$
\end{proposition} 

\begin{proof}
By duality and after renaming the $L^2$ based functions it is enough to consider the estimate
\begin{equation}\label{f1nonlinear}
\iiiint\frac{f(\xi_1,\tau_1)g(\xi-\xi_1,\tau-\tau_1)h(\xi,\tau)\langle \xi \rangle^{s_0+a}d\xi_{1}d\tau_{1}d\xi d\tau}{{\langle \xi_1 \rangle}^{s_{1}}{\langle \xi-\xi_{1} \rangle}^{s_{0}}{\langle \tau_{1}-\xi_{1} \rangle}^{b}{\langle \tau+\xi^2\rangle}^{b}{\langle \tau-\tau_1+(\xi-\xi_1)^2\rangle}^{b}}\lesssim \|f\|_{L^2}\|g\|_{L^2}\|h\|_{L^2}.
\end{equation}
Setting
$$\lambda_1=\tau_1-\xi_1$$
$$\lambda_2=\tau+\xi^2$$
$$\lambda_3=\tau-\tau_{1}+(\xi-\xi_{1})^2$$
we notice that $\lambda_3-\lambda_2+\lambda_1=\xi_1(\xi_1-2\xi-1)$ while if we change variables $\xi\rightarrow \xi-\frac{1}{2}$ inside the integral the identity becomes
$$\lambda_3-\lambda_2+\lambda_1=\xi_1(\xi_1-2\xi).$$ 
We first consider the resonant cases R1 and R2 where $|\xi_1|\lesssim 1$ and $|\xi_1-2\xi|\lesssim 1$ respectively.
\\
\\
{\bf R1:} $|\xi_1|\lesssim 1$.
By applying the Cauchy Schwarz inequality in \eqref{f1nonlinear} first in the $\xi,\tau$ variables and then in the $\xi_1,\tau_1$ variables and integrating the $\tau_1$ integral it is enough to bound the integral
$$\sup_{\xi,\tau}\int\frac{\langle \xi \rangle^{2(s_0+a)}d\xi_{1}}{{\langle \xi_1 \rangle}^{2s_{1}}{\langle \xi-\xi_{1} \rangle}^{2s_{0}}{\langle \tau+\xi^2 \rangle}^{2b}{\langle \tau+(\xi-\xi_1)^2-\xi_1\rangle}^{4b-1}}.$$
Since $4b-1<2b$ this reduces to
\begin{equation}\label{red}
\sup_{\xi}\int\frac{\langle \xi \rangle^{2(s_0+a)}d\xi_{1}}{{\langle \xi_1 \rangle}^{2s_{1}}{\langle \xi-\xi_{1} \rangle}^{2s_{0}}{\langle \xi_1(\xi_1-2\xi) \rangle}^{4b-1}}
\end{equation}
after the aforementioned change of variables $\xi\rightarrow \xi-\frac{1}{2}$. In the case that $|\xi|\lesssim 1$ there is nothing to prove so we consider the subcase that $|\xi|\gg 1$. It is enough then to bound
$$\sup_{\xi}\int_{|\xi_1|\lesssim 1} \frac{\langle \xi \rangle^{2a}}{\langle \xi \xi_1\rangle^{4b-1}}d\xi_1 \lesssim \sup_{\xi}\int_{|\xi_1|\lesssim 1} \frac{\langle \xi \rangle^{2a}}{\langle \xi\rangle^{4b-1}|\xi_1|^{4b-1}}d\xi_1$$
which for any $a<2b-\frac12$ is majorized by
$$\int_{|\xi_1|\lesssim 1} \frac{1}{|\xi_1|^{4b-1}}d\xi_1$$
which is bounded for any $b<\frac12$.
\\
\\
{\bf R2:} $|\xi_1-2\xi|\lesssim 1$. In this case $\langle \xi_1 \rangle \sim \langle \xi \rangle$ and $\langle \xi_1-\xi \rangle=\langle \xi_1-2\xi+\xi \rangle\sim \langle \xi \rangle$.
Then \eqref{red} becomes
\begin{multline*}\label{red_sec}
\sup_{\xi}\int_{|\xi_1-2\xi|\lesssim 1}\frac{\langle \xi \rangle^{2a - 2s_1}d\xi_{1}}{{\langle \xi_1(\xi_1-2\xi) \rangle}^{4b-1}}\lesssim \sup_{\xi}\int_{|\xi_1-2\xi|\lesssim 1}\frac{\langle \xi \rangle^{2a-2s_1}d\xi_{1}}{|\xi_1|^{4b-1}|\xi_1-2\xi| ^{4b-1}}\\ \lesssim \sup_{\xi}\langle \xi \rangle^{2a - 2s_1-4b+1}\int_{|\xi_1-2\xi|\lesssim 1}\frac{d\xi_{1}}{|\xi_1-2\xi| ^{4b-1}}\lesssim 1
\end{multline*}
for any $b<\frac12$ and any $a<s_1+2b-\frac12$.
\vskip 0.05in
\noindent
We now consider the non-resonant frequencies noting that in this case $$\max_{i=1,2,3}|\lambda_{i}| \gtrsim \langle \xi_1 \rangle \langle \xi_1-2\xi \rangle.$$ We will consider the cases that $\lambda_1=\max_{i=1,2,3}|\lambda_{i}|$ and that  $\lambda_2=\max_{i=1,2,3}|\lambda_{i}|$. The case that  $\lambda_3=\max_{i=1,2,3}|\lambda_{i}|$ is almost identical and it is omitted.
\vskip0.05in
\noindent
{\bf Case A:} $|\lambda_2|=\max_{i=1,2,3}|\lambda_{i}| \gtrsim \langle \xi_1 \rangle \langle \xi_1-2\xi \rangle $. We have three subcases.
\\
\\
Case 1: $|\xi_1|\gg |\xi|$. In this case $|\lambda_2| \gtrsim \langle \xi_1 \rangle^2$ and $\langle \xi_1 \rangle \sim \langle \xi-\xi_1 \rangle$. 
By applying Cauchy Schwarz as before in estimate \eqref{f1nonlinear} it is enough to bound
$$\sup_{\xi,\lambda_2}\langle \lambda_2\rangle^{-2\epsilon}\iint\frac{\langle \xi \rangle^{2s_0+2a}d\xi_1 d\lambda_1}{{\langle \xi_1 \rangle}^{2s_{1}+2s_{0}+4b-4\epsilon}{\langle \lambda_1 \rangle}^{2b}{\langle \lambda_3 \rangle}^{2b}}.$$
We now change variables from $\xi_1$ to $\lambda_3$ (for fixed $\xi,\lambda_1,\lambda_2$) using $\lambda_3-\lambda_2+\lambda_2=\xi_1^2-2\xi\xi_1$ to obtain $d\lambda_3\sim \langle \xi_1 \rangle d\xi_1$. Thus the integral is majorized by
$$\sup_{\xi,\lambda_2}\langle \xi \rangle^{2a-1-2s_1-4b+4\epsilon} \langle \lambda_2\rangle^{-2\epsilon}\iint_{|\lambda_1|,|\lambda_3|\lesssim |\lambda_2|}\frac{d\lambda_3 d\lambda_1}{\langle \lambda_1 \rangle^{2b}{\langle \lambda_3 \rangle}^{2b}}
\lesssim \sup_{\xi,\lambda_2}\langle \xi \rangle^{2a-1-2s_1-4b+4\epsilon} \langle \lambda_2\rangle^{2-4b-2\epsilon}$$
which is finite for $\frac12 - b > 0$ sufficiently small and any $a<s_1+2b+\frac12$ .
\\
\\
Case 2: $|\xi|\gg |\xi_1|$. In this case $|\lambda_2|\gtrsim \langle \xi_1 \rangle  \langle \xi\rangle$ and $\langle \xi \rangle \sim \langle \xi-\xi_1 \rangle$. We thus have to bound
$$\sup_{\xi,\lambda_2}\langle \lambda_2\rangle^{-2\epsilon}\iint_{|\xi_1|\lesssim |\xi|}\frac{\langle \xi \rangle^{2a-2b+2\epsilon}d\xi_1 d\lambda_1}{{\langle \xi_1 \rangle}^{2s_{1}+2b-2\epsilon}{\langle \lambda_1 \rangle}^{2b}{\langle \lambda_3 \rangle}^{2b}}.$$
If we dismiss $\lambda_3$ and integrate
$$\int_{|\lambda_1|\lesssim |\lambda_2|}\frac{d\lambda_1}{\langle \lambda_1 \rangle^{2b}} \lesssim |\lambda_2|^{1-2b}$$ 
we need to bound
$$\sup_{\xi,\lambda_2}\langle \xi \rangle^{2a-2b+2\epsilon}\langle \lambda_2\rangle^{1-2b-2\epsilon}\int_{|\xi_1|\lesssim |\xi|}\frac{d\xi_1 }{\lb \xi_1 \rb ^{2s_{1}+2b-2\epsilon}}\lesssim \sup_{\xi,\lambda_2}\langle \xi \rangle^{2a-4b+1-2s_1+4\epsilon}\langle \lambda_2\rangle^{1-2b-2\epsilon}$$
where the last inequality follows because $2s_1+2b<1$. It is now clear that if we pick $b=\frac12-$ and any $a<s_1+2b-\frac12$ the supremum is bounded.
\vskip 0.05in
\noindent
Case 3: $|\xi|\sim |\xi_1|$.
\vskip 0.05in
\noindent
Case 3a: $|\xi_1-2\xi| \gtrsim |\xi|$. In this case notice that since $s_0\leq 0$ we have that $\frac{1}{\langle \xi-\xi_1\rangle^{s_0}}\lesssim \langle \xi \rangle^{-s_0}$ and we need to bound
$$\sup_{\xi,\lambda_2}\langle \xi \rangle^{2a-4b-2s_1+4\epsilon}\langle \lambda_2\rangle^{-2\epsilon}\iint_{|\xi_1|\sim |\xi|}\frac{d\xi_1 d\lambda_1}{{\langle \lambda_1 \rangle}^{2b}{\langle \lambda_3 \rangle}^{2b}}.$$
Now dismissing $\lambda_3$ and using 
$$\int_{|\lambda_1|\lesssim |\lambda_2|}\frac{d\lambda_1}{\langle \lambda_1 \rangle^{2b}} \lesssim |\lambda_2|^{1-2b}$$ and $\int_{|\xi_1|\sim |\xi|}d\xi_1\lesssim |\xi|$ we obtain the desired bound for any $b=\frac12-$ and any  $a<s_1+2b-\frac12$.
\vskip 0.05in
\noindent
Case 3b: $|\xi-\xi_1| \gtrsim |\xi|$. This case is identical with case 3a; the only new ingredient is
$$\int_{|\xi_1|\sim |\xi|}\frac{d\xi_1}{\langle \xi_1-2\xi\rangle^{2b-2\epsilon}} \lesssim |\xi|^{1-2b+2\epsilon}.$$ Again any $b=\frac12-$ and any  $a<s_1+2b-\frac12$ works.
\vskip0.05in
\noindent
{\bf Case B:} $|\lambda_1|=\max_{i=1,2,3}|\lambda_{i}| \gtrsim \langle \xi_1 \rangle \langle \xi_1-2\xi \rangle $. We have three subcases.
\\
\\
Case 1: $|\xi|\gg |\xi_1|$. Here we apply Cauchy Schwarz inequality first in the $\xi_1,\tau_1$ variables
 and then in the $\xi,\tau$ variables. It is sufficient to bound
$$\sup_{\xi_1,\lambda_1}\langle \xi_1\rangle^{-2s_1-2b+2\epsilon}\langle \lambda_1\rangle^{-2\epsilon}\iint_{|\xi_1|\lesssim |\xi|}\frac{\langle \xi \rangle^{2a-2b+2\epsilon}d\xi d\lambda_2}{{\langle \lambda_2 \rangle}^{2b}{\langle \lambda_3 \rangle}^{2b}}.$$ 
Since $a<b$ we estimate $\langle \xi \rangle^{2a-2b+2\epsilon}\lesssim \langle \xi_1 \rangle^{2a-2b+2\epsilon}$ to obtain
$$\sup_{\xi_1,\lambda_1}\langle \xi_1\rangle^{2a-2s_1-4b+4\epsilon}\langle \lambda_1\rangle^{-2\epsilon}\iint_{|\xi_1|\lesssim |\xi|}\frac{d\xi d\lambda_2}{{\langle \lambda_2 \rangle}^{2b}{\langle \lambda_3 \rangle}^{2b}}.$$ 
Since $$\lambda_3-\lambda_2+\lambda_1=\xi_1^2-2\xi_1\xi,$$ for fixed $\lambda_1,\xi_1,\lambda_2$ we have that $d\lambda_3 \sim \langle \xi_1 \rangle d\xi$ (since $|\xi_1|\gg 1$ in the non-resonant case).
We thus need to bound
$$\sup_{\xi_1,\lambda_1}\langle \xi_1\rangle^{2a-2s_1-4b-1+4\epsilon}\langle \lambda_1\rangle^{-2\epsilon}\iint_{|\lambda_2|,|\lambda_3|\lesssim |\lambda_1|}\frac{d\lambda_3 d\lambda_2}{{\langle \lambda_2 \rangle}^{2b}{\langle \lambda_3 \rangle}^{2b}}$$ 
and as above we are done if $b=\frac12-$ and $a<s_1+2b+\frac12$. 
\\
\\
Case 2: $|\xi_1|\gg |\xi|$. By the Cauchy Schwarz inequality, it is enough to bound
$$\sup_{\xi_1,\lambda_1}\langle \lambda_1\rangle^{-2\epsilon}\iint\frac{\langle \xi \rangle^{2s_0+2a}d\xi d\lambda_2}{{\langle \xi_1 \rangle}^{2s_1+2s_{0}+4b-4\epsilon}{\langle \lambda_2 \rangle}^{2b}{\langle \lambda_3 \rangle}^{2b}}\lesssim \sup_{\xi_1,\lambda_1}\langle \lambda_1\rangle^{-2\epsilon}\iint\frac{\langle \xi \rangle^{2a}d\xi d\lambda_2}{{\langle \xi_1 \rangle}^{2s_1+2s_{0}+4b-4\epsilon}{\langle \lambda_2 \rangle}^{2b}{\langle \lambda_3 \rangle}^{2b}}$$
since $s_0\leq 0$.
Changing variables as before we have that $d\lambda_3\sim \langle \xi_1\rangle d\xi$ and thus the integrals are bounded for any $a<s_1+s_0+2b+\frac12$.
\vskip 0.05in
\noindent
Case 3: $|\xi|\sim |\xi_1|$.
\vskip 0.05in
\noindent
Case 3a: $|\xi_1-2\xi| \gtrsim |\xi|$. By Cauchy Schwarz first in the $(\xi_1,\tau_1)$ variables and then in the $(\xi,\tau)$ variables, we need to bound
$$\sup_{\xi_1,\lambda_1}\langle \lambda_1\rangle^{-2\epsilon}\iint_{|\xi|\sim |\xi_1|}\frac{\langle \xi_1 \rangle^{2s_0+2a-2s_1-4b+4\epsilon}d\xi d\lambda_3}{{\langle \xi-\xi_1 \rangle}^{2s_{0}}{\langle \lambda_2 \rangle}^{2b}{\langle \lambda_3 \rangle}^{2b}}.$$
But ${\langle \xi-\xi_1\rangle}^{-2s_0}\lesssim \langle \xi \rangle^{-2s_0}$. Thus by dismissing $\lambda_2$, integrating $\lambda_3$ and using the crude estimate $\int_{|\xi|\sim |\xi_1|}d\xi\lesssim |\xi_1|$ we obtain the bound for any $b=\frac12-$ and $a<s_1+2b-\frac12$.
\vskip 0.05in
\noindent
Case 3b: $|\xi-\xi_1| \gtrsim |\xi|$. In this case it is enough to bound
$$\sup_{\xi_1,\lambda_1}\langle \lambda_1\rangle^{-2\epsilon}\iint_{|\xi|\sim |\xi_1|}\frac{\langle \xi_1 \rangle^{2a-2s_1-2b+2\epsilon}d\xi d\lambda_3}{\langle \xi_1-2\xi\rangle^{2b-2\epsilon}{\langle \lambda_2 \rangle}^{2b}{\langle \lambda_3 \rangle}^{2b}}.$$
If we dismiss $\lambda_2$, integrate $\lambda_3$ and use the estimate
$$\int_{|\xi|\sim |\xi_1|}\frac{d\xi }{\langle \xi_1-2\xi\rangle^{2b-2\epsilon}}\lesssim |\xi_1|^{1+2\epsilon-2b}$$
we obtain the desired bound for any $b=\frac12-$ and $a<s_1+2b-\frac12$.

\end{proof}

\begin{remark}
Our proof of the previous two propostions holds for $b=\frac12-$, $s_0\geq -\frac14$ with any $s_1>-\frac12$. Notice though that for $b=\frac13$, $s_0=0$, and $a=0$, our results restrict $s_1>-\frac16$. However, we must take $b= \frac13$ for our globalizing method to succeed. A more refined application of the Cauchy Schwarz method (in effect a case-by-case implementation of the right change of variables) can improve the previous two Propositions, achieving $s_1\geq -\frac12$ even in the case that $b=\frac13$. Here we show how one can achieve this for the resonant cases. For the non-resonant cases the procedure is similar and it is omitted. The interested reader can easily fill in the details of the argument.

 Note that in Proposition \ref{KG_bilinear} the resonant case covers the case that $b=\frac13$ and $s\geq -\frac12$. In Proposition \ref{KG_bilinearI} the first case follows for any $b<\frac12$ but in the second case where $|\xi_1|\gg 1$ and $|\xi-2\xi|\lesssim 1$ we proceed  as follows (recall that $a=0$ and $s_0=0$):
We set
$$\lambda_1=\tau_1-\xi_1,$$
$$\lambda_2=\tau+\xi^2,$$
$$\lambda_3=\tau-\tau_{1}+(\xi-\xi_{1})^2$$
and notice that it is enough to bound
 \begin{equation}
\iiiint\frac{f(\xi_1,\tau_1)g(\xi-\xi_1,\tau-\tau_1)h(\xi,\tau)\langle \xi \rangle^{-s_1}d\xi_{1}d\tau_{1}d\xi d\tau}{{\langle \lambda_1 \rangle}^{b}{\langle \lambda_2\rangle}^{b}{\langle \lambda_3\rangle}^{b}}\lesssim \|f\|_{L^2}\|g\|_{L^2}\|h\|_{L^2}.
\end{equation}
Now consider the case that $|\lambda_1|=\max_{i=1,2,3}|\lambda_{i}|$. The other two cases are treated similarly. By the Cauchy Schwarz inequality first  in the $\xi_1,\tau_1$ variables
 and then in the $\xi,\tau$ variables  we need to bound
 $$\sup_{\xi_1,\lambda_1}\langle \xi_1\rangle^{-2s_1}\langle \lambda_1\rangle^{-2b}\iint\frac{d\xi d\lambda_2}{{\langle \lambda_2 \rangle}^{2b}{\langle \lambda_3 \rangle}^{2b}}.$$ 
 But since $$\lambda_3-\lambda_2+\lambda_1=\xi_1^2-2\xi_1\xi$$ for fixed $\lambda_1,\xi_1,\lambda_2$ we have that $d\lambda_3 \approx \langle \xi_1 \rangle d\xi$. Thus we need to bound
 $$\sup_{\xi_1,\lambda_1}\langle \xi_1\rangle^{-1-2s_1}\langle \lambda_1\rangle^{-2b}\iint_{|\lambda_2|,|\lambda_3|\lesssim |\lambda_1|}\frac{d\lambda_3 d\lambda_2}{{\langle \lambda_2 \rangle}^{2b}{\langle \lambda_3 \rangle}^{2b}}\lesssim \sup_{\xi_1,\lambda_1}\langle \xi_1\rangle^{-1-2s_1}\langle \lambda_1\rangle^{2-6b}\lesssim 1$$
 for any $\frac13 \leq b<\frac12$ and any $s_1\geq -\frac12$.
 
\end{remark}

\section{Local Existence \& Smoothing}\label{lwp}

We first prove that the map $\Gamma=(\Gamma_1,\Gamma_2)$, defined by 
\begin{equation} \label{gamma}\begin{array}{l}
\Gamma_1(u,n)(t)=\eta(t) W_0^t\big(u_0^e, g  \big) -i \eta(t) \int_0^t e^{i(t- t^\prime)\Delta}   F(u,n) \d t^\prime +i\eta(t) W_0^t\big(0,  q  \big), \\
\Gamma_2(u,n)(t)= \eta(t) V_0^t\big(\phi_{\pm}, h \big) + \frac12\eta(t) (n_+ +  n_-) - \frac12\eta(t) V_0^t(0,z),  
\end{array}
\end{equation}
has a fixed point in $X^{s_0,b}(\R)\times Y^{s_1,b}(\R)$. Recall that 
$F$, $q$, $n_\pm$, and $z$ are defined in \eqref{eq:def_duhamel_compnts}.

We begin with $\Gamma_1$. To see that the Duhamel term is bounded in $X^{s_0,b}\times Y^{s_1,b}$, we use following bounds. 
Combining \eqref{eq:xs2}, \eqref{eq:xs3}, and Proposition \ref{KG_bilinearI}, we obtain
\begin{multline*}
  \left\| \eta(t) \int_0^t e^{i(t- t')\Delta}   F(u,n) \d t' \right\|_{X^{s_0,b}} \lesssim \|F(u,n)\|_{X^{s_0,-\frac12+}} \\
  \lesssim T^{\frac12-b-} \|un\|_{X^{s_0,-b}}\lesssim T^{\frac12-b-} \|u\|_{X^{s_0,b}} \|n\|_{Y^{s_1,b}}.
\end{multline*}
To bound the linear part of $\Gamma_1$, recall that  
\[
\eta(t) W_0^t\big(u_0^e, g  \big)+i\eta(t) W_0^t\big(0,  q  \big)= \eta(t)e^{it\Delta} u_0^e+\eta(t) W_0^t\big( 0, g-p+iq \big). 
\]
By \eqref{eq:xs1} and our choice of $u_0^e$, we have
\[
\|\eta(t) e^{it\Delta} u_0^e\|_{X^{s_0,b}} \lesssim \|u_0^e\|_{H^{s_0}}\lesssim \|u_0\|_{H^{s_0} (\R^+)}.
\]
Using Lemma \ref{schr_ibvp_xsb} and Lemma \ref{char_func_lemma}  we have
\begin{multline} \label{eq:temp1}
 \| \eta(t) W_0^t\big( 0, g-p+iq \big)(t)\|_{X^{s_0,b}}
\lesssim \|(g-p+iq)\chi_{(0,\infty)} \|_{H^{\frac{2s_0+1}4}_t(\R)} \\ \lesssim \|g-p\|_{H^{\frac{2s_0+1}4}_t(\R^+)} + \|q\|_{H^{\frac{2s_0+1}4}_t(\R^+)}
\lesssim \|g \|_{H^{\frac{2s_0+1}4}_t(\R^+)} +\|p\|_{H^{\frac{2s_0+1}4}_t(\R )}+\|q\|_{H^{\frac{2s_0+1}4}_t(\R )}.
\end{multline}
By Lemma~\ref{schr_ivp_cont} and the definition of $p$ (see \eqref{eq:SchrIBVPformula}), we have
$$
\|p\|_{H^{\frac{2s_0+1}4}_t(\R )} \lesssim \|u_0\|_{H^{s_0}(\R^+)}.
$$
Recalling the definition of $q$ from \eqref{eq:def_duhamel_compnts} and combining  Proposition~\ref{schr_duhamel_est}, \eqref{eq:xs3}, and Proposition \ref{KG_bilinearI} yields 
$$
\|q\|_{H^{\frac{2s_0+1}4}_t(\R )} 
\lesssim T^{ \frac12-b-} \| u\|_{X^{s_0,b}} \|n\|_{Y^{s_1,b}}.
$$
Combining all these estimates, we obtain
\[
\|\Gamma_1 (u,n)\|_{X^{s_0,b}}\lesssim \|u_0\|_{H^{s_0}(\R^+)}+ \|g \|_{H^{\frac{2s_0+1}4}_t(\R^+)} + T^{ \frac12-b-} \| u\|_{X^{s_0,b}} \|n\|_{Y^{s_1,b}}.
\]

For the Duhamel term in $\Gamma_2$, we use the $Y^{s_1,b}_\pm$ analogs of \eqref{eq:xs2} and \eqref{eq:xs3} along with Proposition~\ref{KG_bilinear} to get
\[
\|\eta(t) (n_+ +  n_-)\|_{Y^{s_1,b}}\leq \|\eta(t)  n_+  \|_{Y^{s_1,b}_+}+\|\eta(t)  n_-  \|_{Y^{s_1,b}_-}  
\lesssim T^{\frac12-b-} \|u\|_{X^{s_0,b}}^2.
\]
For the linear part of $\Gamma_2$, first write
$$
\eta(t)V_0^t(\phi_\pm, h)-\frac12 \eta(t)V_0^t(0, z)=   \frac12 \eta(t) \left[e^{  t \partial_x} \phi_{+}+ e^{- t \partial_x} \phi_{-}\right] + \eta(t) V_0^t(0, h-r- z/2).
$$
Using Lemma \ref{ibvp_xsb} yields
\begin{multline*} \label{eq:temp12}
 \| \eta(t) V_0^t\big( 0,  h-r- z/2 \big)(t)\|_{Y^{s_1,b}}
\lesssim \|(h-r-z/2)\chi \|_{H^{s_1}_t(\R)} \\ \lesssim \|h-r\|_{H^{s_1}_t(\R^+)} + \|z\|_{H^{s_1}_t(\R^+)}
\lesssim \|h \|_{H^{s_1}_t(\R^+)} +\|r\|_{H^{s_1}_t(\R )}+\|z\|_{H^{s_1}_t(\R )}.
\end{multline*}
By Lemma~\ref{ivp_cont}, we have
\[
\|r\|_{H^{s_1}_t(\R )} \lesssim \|\phi_+\|_{H^{s_1}_t(\R )}+\|\phi_-\|_{H^{s_1}_t(\R )}\lesssim  \|n_0\|_{H^{s_1}_x(\R^+)} +\|n_1\|_{ H^{s_1 -1 }_x(\R^+)}.
\]
Finally, by combining  Proposition~\ref{wave_duhamel_est}, Proposition~\ref{prop:smooth23}, \eqref{eq:xs3}, and Proposition~\ref{KG_bilinear} we have
\[
\|z\|_{H^{s_1}_t(\R )} 
\lesssim T^{ \frac12-b-} \| u\|_{X^{s_0,b}}^2.
\]
Combining these estimates with the wave version of \eqref{eq:xs1},  we obtain
\[
\|\Gamma_2 (u,n)\|_{Y^{s_1,b}}\lesssim \|h \|_{H^{s_1}_t(\R^+)}+ \|n_0\|_{H^{s_1}_x(\R^+)} +\|n_1\|_{H^{s_1-1}_x(\R^+)} + T^{ \frac12-b-} \| u\|_{X^{s_0,b}}^2.
\]

The differences can be estimated similarly. Therefore, for $T$ sufficiently small, $\Gamma = (\Gamma_1, \Gamma_2)$ has a fixed point $(u,n) \in X^{s_0,b}\times Y^{s_1,b}$.

Next we show that $u\in C^0_t H^{s_0}_x([0,T]\times \R)$. Continuity in $H^{s_0}$ of the  first term of $\Gamma_1$ follows from continuity of the linear Schr\"odinger flow on $\R$. That of the third term is obtained using the embedding $X^{s_0,\frac12+} \hookrightarrow C_t^0 H^{s_0}$ (which holds for $b > \frac12$) and then using \eqref{eq:temp1}. For the Duhamel integral term, it follows from the embedding  $X^{s_0,\frac12+}\hookrightarrow C^0_tH^{s_0}_x$
 and \eqref{eq:xs2} together with Proposition~\ref{KG_bilinearI}. Similarly, we have $u\in C^0_x H^{\frac{2s_0+1}4}_t( \R \times [0,T])$ by Lemma~\ref{schr_ivp_cont}, Propositions~\ref{schr_duhamel_est} and \ref{schr_duhamel_corr}, and Lemma~\ref{schr_ibvp_cont}. The corresponding results for $n$ are proved similarly. 
Continuous dependence on initial and boundary data follows from the fixed point argument and the estimates given above.

The smoothing result is established by estimating the nonlinear terms in $X^{s_0,b} \times Y^{s_1,b}$ spaces just as in the local theory and exploiting the slack in the nonlinear estimates of Propositions \ref{KG_bilinear} and \ref{KG_bilinearI}. We omit the details.

\section{Uniqueness}\label{uni}

In this section, we consider uniqueness of solutions to the KGS system \eqref{eq:KGS}. First we discuss uniqueness for $s_0 > \frac12$. 

Suppose we obtain two solutions, say $(u_1, n_1)$ and $(u_2,n_2)$, to the system \eqref{eq:KGS} with the same initial and boundary conditions. Since $n_1$ and $n_2$ can be defined on the whole line, we may define $n_i^{\pm} = n_i \pm i D^{-1}(n_i)_t$. On the positive half-line, the $n^\pm_i$ satisfy
\[ i (n^\pm_i)_t = \pm D n^\pm_i \mp D^{-1}(|u|^2).\]  
Let $v = u_1 - u_2$ and $m^\pm = n_1^\pm - n_2^\pm$. These functions are defined on $\R \times \R^+$, and on $\R^+ \times \R^+$ they satisfy
\begin{equation*}
 \begin{cases}
  iv_t + \Delta v = \frac12\Bigl( u_1(m^+ + m^-) + v(n_2^+ +n_2^-) \Bigr) \\
  i m^\pm_t = \pm D m^\pm \mp D^{-1}\Bigl( u_1\overline{v} + \overline{u_2}v \Bigr).  
 \end{cases}
\end{equation*}
The initial and boundary conditions are all zero, in (at least) the $L^2$ sense.
Then we compute
\begin{align*}
 \partial_t \|v\|_{L^2_x(\R^+)}^2 &= \Im \int_0^\infty u_1 (m^+ + m^-) \overline{v} \d x \\
 \partial_t \|m^\pm\|_{L^2_x(\R^+)}^2 &=  \mp 2 \Im \int_0^\infty D^{-1}(u_1\overline{v} + \overline{u_2}v) \overline{m}^\pm \d x. 
\end{align*}

Integrating in time, we obtain
\begin{align*}
 \|v(t)\|_{L^2_x(\R^+)}^2 &\lesssim \|u_1\|_{L^\infty_{[0,t]}L^\infty_x} \int_0^t \Bigl[ \|m^+\|_{L^2_x(\R^+)} + \|m^-\|_{L^2_x(\R^+)} \Bigr]  \|v\|_{L^2_x(\R^+)} \d t'\\
 \|m^\pm\|_{L^2_x(\R^+)}^2 &\lesssim \Bigl[ \|u_1\|_{L^\infty_{[0,t]}L^\infty_x} +  \|u_2\|_{L^\infty_{[0,t]}L^\infty_x} \Bigr] \int_0^t \|m^\pm\|_{L^2_x(\R^+)}  \|v\|_{L^2_x(\R^+)} \d t'. 
\end{align*}
Using the embedding $H^{\frac12+} \hookrightarrow L^\infty$ and the local theory bounds, the $L^\infty$ norms of $u_1$ can be bounded by constants. Thus we get the inequality
\[ \Bigl[  \|v(t)\|_{L^2_x(\R^+)} + \|m^+\|_{L^2_x(\R^+)} + \|m^-\|_{L^2_x(\R^+)} \Bigr]^2 \lesssim \int_0^t \Bigl[  \|v(t)\|_{L^2_x(\R^+)} + \|m^+\|_{L^2_x(\R^+)} + \|m^-\|_{L^2_x(\R^+)} \Bigr]^2 \d t'.\] 
Gr\"onwall's inequality then implies that $v = m^\pm =0$; i.e. $u_1=u_2$ and $v_1 = v_2$ on the right half-line.

This establishes uniqueness of solutions in $H^{\frac12+} \times L^2$. It remains to address uniqueness of rougher solutions. Consider initial and boundary data
\[ (u_0, n_0, n_1,g,h) \in  H^{s_0}(\R^+) \times H^{s_1}(\R^+) \times {H}^{s_1 - 1}(\R^+) \times H^{\frac{2s_0 + 1}4}(\R^+) \times H^{s_1}(\R^+).\]
Suppose first that $s_0 \in (0,\frac12)$ and $s_1 = -\frac12+ $. In addition suppose $u_0^e$ and $\wt{u_0}^e$ are two $H^{s_0}(\R)$ extensions of $u_0$, and $(n_0^e, n_1^e)$ and $(\wt{n_0}^e,\wt{n_1}^e)= (n_0^{\text{odd}}, n_1^{\text{odd}})$ are $H^{s_1}(\R)\times H^{s_1-1}(\R)$ extensions of $(n_0,n_1)$. Note that for the wave data, we choose one extension to be specifically the odd extension, for reasons which are explained below.

Let $(u,n)$ and $(\wt{u}, \wt{n})$ be the corresponding solutions to the fixed point equation. Take a sequence ${u_{0,k}}$ in $H^{\frac12+}(\R^+)$ which converges to $u_0$ in $H^{s_0}(\R^+)$. Let $u_{0,k}^e$ and $\wt{u_{0,k}}^e$ be $H^{\frac12+}$ extensions of $u_{0,k}$ which converge to $u_0^e$ and $\wt{u_0}^e$ respectively in $H^{s_0-}(\R)$. Such extensions exist by Lemma \ref{approx_lemma} below.

For the wave data, which has a component in $H^{s_1-1}$, Lemma \ref{approx_lemma} is not available. We proceed slightly differently, employing the odd extension. Take a sequence $(n_{0,k}^e, n_{1,k}^e) \in L^2(\R)\times H^{-\frac12+}(\R)$ which converges to $(n_0^e, n_1^e)$ in $H^{s_1}(\R)\times H^{s_1-1}(\R)$.
Define a second sequence $(\wt{n_{0,k}}^e, \wt{n_{1,k}}^e) \in L^2(\R)\times H^{-\frac12+}(\R)$ by
\[ (\wt{n_{0,k}}^e, \wt{n_{1,k}}^e) = \Bigl( (\chi n_{0,k}^e)^\text{odd},(\chi n_{1,k}^e)^\text{odd}  \Bigr).\] 
Then $(n_{0,k}^e, n_{1,k}^e) = (\wt{n_{0,k}}^e, \wt{n_{1,k}}^e)$ on the positive half-line, and $(\wt{n_{0,k}}^e, \wt{n_{1,k}}^e)$ converges to $(\wt{n_0}^e, \wt{n_1}^e)$ in $H^{s_1}(\R) \times H^{s_1-1}(\R)$.

Using the local theory in $H^{\frac12+} \times L^2 \times H^{-1}$, we arrive at corresponding sequences of solutions $(u_k, n_k)$ and $(\wt{u_k}, \wt{n_k})$. Since their initial data is equal on the right half-line, the uniqueness result above implies that $(u_k, n_k)$ and $(\wt{u_k}, \wt{n_k})$ are equal on $\R^+$ on their common interval of existence. Furthermore, $(u_k, n_k)$ converges to $(u,n)$ and $(\wt{u_k}, \wt{n_k})$ converges to $(\wt{u}, \wt{n})$ by the local well--posedness theory we have established in Section \ref{lwp}. Thus, if we can show that the common interval of existence is nontrivial, we will have uniqueness. 

\emph{A priori}, the interval of existence is inversely proportional to the $H^{\frac12+} \times L^2 \times H^{-1}$ norms of the initial data, which are growing as $k$ increases. This means that the time of existence goes to zero as $k \to \infty$. However, using the smoothing, we can take the time of existence proportional to the data in the $H^{s_0} \times H^{-\frac12+} \times H^{-\frac32+}$ norm, which is bounded as desired. This works directly for $s_0 > 0$. Iterating this argument, we can obtain uniqueness for $s_0 \in (-\frac14, 0]$ as well. 

\begin{lemma}{\cite{ET1}}\label{approx_lemma}
 Fix $-\frac12 < s < \frac12$ and $k > s$. Let $p \in H^s(\R^+)$ and $q \in H^k(\R^+)$. Let $p^e$ be an $H^s$ extension of $p$ to $\R$. Then there is an $H^k$ extension $q^e$ of $q$ to $\R$ such that 
 \[ \| p^e - q^e \|_{H^r(\R)} \lesssim \| p-q\|_{H^s(\R^+)} \quad \text{ for } r < s. \] 
\end{lemma}

\section{Global Existence}\label{globalsec}

To begin, we establish conservation of $\|u\|_{L^2}$ for the Klein-Gordon Schr\"odinger system \eqref{eq:KGS} with Schr\"odinger boundary data $g=0$. Multiply the Schr\"odinger evolution equation by $\overline{u}$ to obtain
\begin{align*}
 \partial_t \|u\|_{L^2_x(\R^+)}^2 = 2 \Re \int_0^\infty u_t \overline{u} \d x &= -  2 \Im\int_0^\infty u_{xx}\overline{u} \d x \\
 &=  -2 \Im \overline{u}(0,\cdot) u_x(0,\cdot) = -2 \Im \overline{g}(\cdot) u_x(0,\cdot). 
\end{align*}
Integrating this equality we arrive at 
\[ \|u(\cdot, t)\|_{L^2_x(\R^+)}^2 = \| u_0 \|_{L^2_x(\R^+)}^2 -2 \Im \int_0^t \overline{g}(t') u_x(0,t') \d t' \; \overset{g=0}{=} \; \| u_0 \|_{L^2_x(\R^+)}^2.\] 
Thus, for $g = 0$, we have conservation of $\|u\|_{L^2_x}$.

To carry out the global existence argument, we assume that the Schr\"odinger part has zero boundary data, and work with the system
\begin{equation} \label{eq:gKGS}
 \begin{cases}
  iu_t + \Delta u = (n+m)u, \quad x,t \in \R^+, \\
  n_{tt} + (1 - \Delta)n = |u|^2, \\
  u(x,0) = u_0 \in L^2 (\R^+) \\
  n(x,0) = n_0(x) \in H^{s_1}(\R^+), \quad n_t(x,0) = n_1(x) \in {H}^{s_1 - 1}(\R^+), \\
  u(0,t) = 0, \quad n(0,t) = 0. 
 \end{cases}
\end{equation}

Here $m$ is the solution to the linear Klein-Gordon initial value problem with zero initial data and boundary data $h(t) \in H^{s_1}(\R^+)$. We note that the $L^2$ conservation is still valid for the new system \eqref{eq:gKGS}.

Then, estimating as in the local theory argument and using the powers of $T$ available from \eqref{eq:xs3}, \eqref{eq:T_power_linear}, Lemmas~\ref{schr_ibvp_cont} and \ref{ibvp_cont} and Proposition~\ref{T_power_est},  we have 
\begin{align*}
 &\| \Gamma_1(u(t),n(t)) \|_{X^{0,b}} \lesssim T^{\frac12  -b} \| u_0 \|_{L^2} + T^{1 - 2b} \| n + m\|_{Y^{s_1, b}} \|u\|_{X^{0,b}}, \\
 &\| \Gamma_2 (u(t),n(t)) \|_{Y^{s_1,b}} \lesssim T^{\frac12 - b} (\|n_0\|_{H^{s_1}} + \|n_1\|_{{H}^{s_1-1}}) + T^{1-2b- \epsilon} \| u\|_{X^{0,b}}^2,
\end{align*}
and 
\begin{align*}
 &\| \Gamma_1(u(t), n(t)) - \Gamma_1(\tilde{u}(t), \tilde{n}(t))\|_{X^{0,b}} \\ 
 &\hspace{2in} \lesssim   T^{1-2b} \Bigl( \|n + m\|_{Y^{s_1,b}} \| u - \tilde{u}\|_{X^{0,b}} + \| n - \tilde{n} \|_{Y^{s_1,b}} \| \tilde{u} \|_{X^{0,b}} \Bigr), \\
 &\| \Gamma_2(u(t), n(t)) - \Gamma_2(\tilde{u}(t), \tilde{n}(t))\|_{Y^{s_1,b}} \\ 
 & \hspace{2in} \lesssim T^{1-2b-\epsilon} \| u -\tilde{u}\|_{X^{0,b}} \Bigl( \| u\|_{X^{0,b}} + \| \tilde{u}\|_{X^{0,b}} \Bigr). 
\end{align*}

Thus, on a ball in $X^{0,b} \times Y^{s_1,b}$ given by
\begin{equation} \label{eq:contractionSpace} \| u\|_{X^{0,b}} \lesssim T^{\frac12-b} \|u_0\|_{L^2} \qquad \| n\|_{Y^{s_1,b}} \lesssim T^{\frac12-b} (\|n_0\|_{H^{s_1}} + \|n_1\|_{H^{s_1-1}}), 
\end{equation}
we can obtain a contraction as long as 
\begin{align}\label{eq:contraction_cond}
\begin{split}
 &T^{\frac32 - 3b} \| u_0\|_{L^2} \lesssim 1,    \\      
 &T^{\frac32 - 3b - \epsilon} \| u_0\|_{L^2} \lesssim 1, \\
 &T^{\frac32 - 3b} (\|n_0\|_{H^{s_1}} + \|n_1\|_{{H}^{s_1-1}} + \|h\|_{H^{s_1}}) \lesssim 1,  \\
 &T^{\frac32 - 3b - \epsilon} \| u_0\|_{L^2}^2 \lesssim (\|n_0\|_{H^{s_1}} + \|n_1\|_{{H}^{s_1-1}}).
 \end{split}
\end{align}

We wish to iterate this process. The spatial $L^2$ norm of the Schr\"odinger part is conserved, so we need not concern ourselves with the growth of $\|u\|_{L^2}$. The boundary data $h$ is also fixed for all time, so we need only concern ourselves with the growth of the spatial $H^{s_1} \times H^{s_1-1}$ norm of $(n,n_t)$.

Suppose that after some time $t$, we have $\| n(t)\|_{H^s_1(\R)} + \|n_t\|_{H^{s_1 - 1}(\R)} \gg \Bigl \lb \|u_0\|_{L^2} + \|h\|_{H^{s_1}} \Bigr\rb^2$. Take this as the new initial time. To satisfy the second inequality in \eqref{eq:contraction_cond} with the optimal $b = \frac13$, we take
\[ T \approx (\|n_0\|_{H^{s_1}} + \|n_1\|_{{H}^{s_1-1}})^{-1/(\frac32 - 3b)} = (\|n_0\|_{H^{s_1}} + \|n_1\|_{{H}^{s_1-1}})^{-2}. \] 
Note that the other constraints in \eqref{eq:contraction_cond} are then automatically satisfied. 

As initial data for the next iteration, we take $(u(T), n^\text{odd}(T), n^\text{odd}_t(T))$. We need to bound the norms of $n^\text{odd}(T)$ and $n^\text{odd}_t(T)$. Recall that
\[  n = \eta_T(t) V_0^t\big(\phi_{\pm}, 0\big) + \frac12\eta(t) (n_+ +  n_-) -\frac12 \eta_T(t) V_0^t(0,z).\]
We assume that we obtained $(u,n)$ by taking odd extensions of the initial data. In this case, $r$, as defined by \eqref{eq:rDef}, is zero because the Klein-Gordon flow preserves oddness.  
 
It remains to control the remaining terms comprising $n$. If $s_1 = 0$, we could proceed directly using estimates similar to those established already.\footnote{Specifically, we could use an estimate similar to Proposition \ref{T_power_est} to control $V_0^t$ in $L^2$. However, a sufficiently strong estimate in $H^{s_1}$, for $s_1 < 0$, does not appear to hold.}
However, to take $s_1 <0$ an additional observation is needed. We note that the remaining term $\frac12\eta(t) (n_+ +  n_-) -\frac12 \eta_T(t) V_0^t(0,z)$ is exactly the solution to the Klein-Gordon on $\Bbb R^+$ with zero initial and boundary conditions and forcing $|u|^2$. On the right-half line, this is the same as the solution to the KGS system \eqref{eq:gKGS} with $(|u|^2)^{\text{odd}}$ forcing. We use this with the fact that the Klein-Gordon flow preserves oddness and the estimate $\| f^\text{odd}\|_{Y^{s,b}_\pm} \lesssim \|f\|_{Y^{s,b}_\pm}$ (which follows from Lemma \ref{char_func_lemma}). This will allow us to eliminate the troublesome $V_0^t$ term entirely. We have
\begin{multline*}
 \| n^\text{odd}(T) \|_{H^{s_1}_x} +  \| n_t^\text{odd}(T) \|_{H^{s_1 -1}_x} \leq
 \| n(0) \|_{H^{s_1}_x} +  \| n_t(0) \|_{H^{s_1 -1}_x} + \\
  + \frac12 \left\| (\wt{n_+} + \wt{n_-})(T) \right\|_{H^{s_1}_x} + \frac12 \left\|\partial_t(\wt{n_+} + \wt{n_-})(T) \right\|_{H^{s_1-1}_x}.
\end{multline*}
Here $\wt{n_\pm}$ denotes the Duhamel integral as defined in \eqref{eq:def_duhamel_compnts}, with the exception that $|u|^2$ is replaced by $(|u|^2)^\text{odd}$. By \cite[Lemma 2.3]{CHT}, the terms on second line above can be bounded by
\[ T^{\frac16 }\| (|u|^2)^\text{odd} \|_{Y^{s_1,-\frac13}} \lesssim T^{\frac16 }\| |u|^2 \|_{Y^{s_1,-\frac13}} \lesssim T^\frac12 \| u_0\|_{L^2}^2.\]
The last inequality above comes from the bound on the size of $u$; see \eqref{eq:contractionSpace}. 

From this point, the argument closes exactly as in \cite{CHT} -- we can iterate this process $m$ times before the norms double, where
\[m \approx \frac{\| n_0\|_{H^{s_1}} + \| n_1\|_{H^{s_1-1}}}{T^\frac12 \| u_0\|^2}.\] 
The time advanced after these iterations is 
\[ mT \approx \frac{T^{\frac12} \bigl(\| n_0\|_{H^{s_1}} + \| n_1\|_{H^{s_1-1}}\bigr)}{ \| u_0\|^2} \approx \frac1{\|u_0\|^2},\]
which is independent of the wave data. Thus the entire process can be iterated to cover intervals of arbitrary length. 

\section{Proofs of Estimates}\label{proofs}

\subsection{Proof of Lemma \ref{ivp_cont}: Kato Smoothing for the Klein-Gordon Flow} \label{ivp_cont_prf}

It suffices to consider evaluation at $x=0$ since Sobolev norms are translation invariant. We may write
\begin{align*}
 \eta(t) e^{\pm t D} g(0,t) &= \eta(t) \int e^{\pm i t \operatorname{sgn}(\xi)\lb \xi \rb} \hat{g}(\xi) \d \xi \\
 \mathcal{F}_t\Bigl(\eta(t) e^{\pm t D} g(0,t) \Bigr) &= \int \hat{\eta}(\tau \mp \operatorname{sgn}(\xi)\lb \xi \rb ) \hat{g}(\xi) \d \xi.
\end{align*}
Using this representation and the fact that $\lb \tau \rb^s \lesssim \lb \xi \rb^s \lb \tau \mp \operatorname{sgn}(\xi)\lb \xi \rb \rb^{|s|}$, we arrive at
\begin{align*}
\| \eta(t) e^{\pm t D} g(0,t)\|_{H^s_t} &= \left\| \lb \tau \rb^s \int \hat{\eta}(\tau \mp \operatorname{sgn}(\xi)\lb \xi \rb ) \hat{g}(\xi) \d \xi\right\|_{L^2_\tau} \\
&\lesssim \left\| \int  \lb \tau \mp \lb \xi \rb \rb^{|s|} \hat{\eta}(\tau \mp \operatorname{sgn}(\xi)\lb \xi \rb ) \lb \xi \rb^s \hat{g}(\xi) \d \xi\right\|_{L^2_\tau}.
\end{align*}
Since $\eta$ is a Schwarz function, Young's inequality implies that this is bounded by $\|g\|_{H^s}$.

\subsection{Proof of Lemma \ref{ibvp_xsb}: $Y^{s,b}$ Bound for the Klein-Gordon Solution on $\Bbb R^+$} \label{ibvp_xsb_prf}

Recall the formulas for $A$ and $B$ given by Lemma \ref{AB}. To bound $A$, let $f(y) = e^{-y} \rho(y)$. This is a Schwarz function.
The space-time Fourier transform of $\eta_T(t) A$ is 
\[ T \int_{-1}^1 \hat{\eta}(T(\tau - \mu)) \mathcal{F}_x\Bigl(f(x\sqrt{1-\mu^2})\Bigr)(\xi) \hat{h}(\mu) \d \mu = T \int_{-1}^1 \hat{\eta}(T(\tau - \mu)) \frac{\hat{f}(\xi/\sqrt{1-\mu^2})}{\sqrt{1-\mu^2}} \hat{h}(\mu) \d \mu.\] 
Since $f$ and $\eta$ are Schwarz functions and $|\mu| \leq 1$, we have the bounds 
\begin{align*} 
\hat{\eta}(T(\tau - \mu)) &\lesssim \lb T(\tau - \mu) \rb^{-|b| - 1/2- } \lesssim T^{-|b|} \lb \tau \rb^{-|b|} \lb T\tau \rb^{ - 1/2 - } \\
\hat{f}(\xi/\sqrt{1-\mu^2}) &\lesssim \lb \xi/\sqrt{1 - \mu^2} \rb^{-2 -|s| -|b|} \lesssim \left( \frac{1-\mu^2}{1-\mu^2 + \xi^2} \right) \lb \xi \rb^{-|s| - |b|}. 
\end{align*}
Thus
\begin{align*}
 \| \eta_T(t) A \|_{Y^{s,b}_\pm} &\lesssim T^{1-|b|} \left\| \lb \xi \rb^{-|b|}\lb \tau \pm |\xi| \rb^b \lb \tau \rb^{-|b|} \lb T\tau \rb^{ - 1/2 -} \int_{-1}^1 \frac{(1-\mu^2)^{1/2}}{1-\mu^2 + \xi^2} \hat{h}(\mu) \d \mu \right\|_{L^2_\xi L^2_\tau} \\
 &\lesssim T^{1-|b|} \left\|  \lb T \tau \rb^{-1/2-} \int_{-1}^1 \frac{(1-\mu^2)^{1/2}}{1-\mu^2 + \xi^2} \hat{h}(\mu) \d \mu \right\|_{L^2_\xi L^2_\tau}\\
 &\lesssim T^{1-|b|} \left\| \lb T \tau \rb^{-1/2-} \int_{-1}^1 (1-\mu^2)^{-1/4} \hat{h}(\mu) \d \mu \right\|_{ L^2_\tau} \\
 & \lesssim T^{1/2-|b|} \| \chi_{[-1,1]} \hat{h}\|_{L^2} \leq \| \chi h \|_{H^s(\R)}.
\end{align*}
To bound $B$, notice that $B =L^t \phi(x)$, where $L^t$ is the Fourier multiplier operator given by $e^{-it \mu  \sqrt{1 + 1/\mu^2}}$ and 
\[ \widehat{\phi}(\mu) =  \widehat{h}\bigl(-\mu \sqrt{1+ 1/\mu^2}\bigr) \frac{1}{\sqrt{1 + 1/\mu^2}}. \] 
If we establish the bound $\| \eta_T(t) L^t \phi\|_{Y^{s,b}} \lesssim T^{1/2-|b|}\|\phi\|_{H^s}$, and note that  
\[ \| \phi \|_{H^s(\R)}^2  = \int_{|z| \geq 1} \lb \sqrt{z^2 - 1} \rb^{2s} \,\widehat{h}^2(z) \d z \leq  \int_{|z| \geq 1}\lb z \rb^{2s} \,\widehat{h}^2(z) \d z \leq \| \chi h \|_{H^s(\R)}^2,\] 
we'll be done. 

To show that $\| \eta_T(t) L^t \phi\|_{Y^{s,b}} \lesssim T^{1/2 - |b|} \|\phi\|_{H^s}$, notice first that 
\[ \mathcal{F}_{x,t}\Bigl(\eta_T(t) L^t \phi\Bigr)(\xi,\tau)  = T \hat{\eta}(T(\tau - \operatorname{sgn}(\xi) \lb \xi \rb) ) \hat{\phi}(\xi).\] 
Write $\eta_T(t) L^t \phi = T(\I + \II)$, where
\[ \hat{\I}(\xi,\tau) = \rho(\xi) \hat{\eta}(T(\tau - \operatorname{sgn}(\xi) \lb \xi \rb)) \hat{\phi}(\xi) \qquad \hat{\II} = \bigl(1-\rho(\xi)\bigr)\hat{\eta}(T(\tau - \operatorname{sgn}(\xi) \lb \xi \rb)) \hat{\phi}(\xi). \] 
By definition $ \| \eta_T(t) L^t \phi\|_{Y^{s,b}} \leq \| \I \|_{Y^{s,b}_-} + \| \II\|_{Y^{s,b}_+}$. Since $\rho$ is supported on $[-1,\infty)$ and $\eta$ is Schwarz function, we have
\begin{align*}
 \|\I\|_{Y^{s,b}_-} &= T \| \lb \xi \rb^s \lb \tau - |\xi| \rb^b \rho(\xi) \hat{\eta}(T(\tau - \operatorname{sgn}(\xi) \lb \xi \rb)) \hat{\phi}(\xi) \|_{L^2_\xi L^2_\tau}\\
 &\lesssim T\| \lb \xi \rb^s \lb \tau - \xi \rb^b \rho(\xi) \lb T( \tau - \lb \xi \rb) \rb^{-1-|b|} \hat{\phi}(\xi) \|_{L^2_\xi L^2_\tau} \\
 &\lesssim T^{1-|b|} \| \lb \xi \rb^s \lb \tau - \xi \rb^b \rho(\xi) \lb T( \tau - \lb \xi \rb) \rb^{-1} \lb  \tau - \lb \xi \rb \rb^{-|b|} \hat{\phi}(\xi) \|_{L^2_\xi L^2_\tau} \\
 &\lesssim T^{1-|b|} \| \lb \xi \rb^s \lb T(\tau - \xi) \rb^{-1} \hat{\phi}(\xi) \|_{L^2_\xi L^2_\tau} \lesssim T^{1/2-|b|} \| \phi\|_{H^s}.
\end{align*}
Similarly, $\| \II\|_{Y^{s,b}_+} \lesssim T^{1/2-|b|} \| \phi\|_{H^s}$.

\subsection{Proof of Lemma \ref{ibvp_cont}: Continuity of Klein-Gordon Solution $\Bbb R^+$} \label{ibvp_cont_prf}

To show that $A \in C_t^0H_x^s$, write $f(y) = e^{-y} \rho(y)$ and notice that 
\begin{align*} 
\hat{A}(\xi, t) = \int_{-1}^1 e^{i \mu t} \frac{\hat{f}(\xi/ \sqrt{1 - \mu^2})}{\sqrt{1-\mu^2}} \hat{h}(\mu) \d \mu 
\end{align*}
Recall also that for $|\mu| \leq 1$, we have
\[ \hat{f}(\xi/ \sqrt{1 - \mu^2}) \lesssim \lb \xi \rb^{-|s|} \frac{1 - \mu^2}{1-\mu^2 + \xi^2}. \]
Therefore
\begin{align*}
 \| A\|_{H^s_x} &= \left\| \lb \xi \rb^s \int_{-1}^1 e^{i \mu t} \frac{\hat{f}(\xi/ \sqrt{1 - \mu^2})}{\sqrt{1-\mu^2}} \hat{h}(\mu) \d \mu \right\|_{L^2_\xi}\\
 &\lesssim \left\|  \int_{-1}^1 e^{i \mu t} \frac{\sqrt{1 - \mu^2}}{1-\mu^2+\xi^2} \hat{h}(\mu) \d \mu \right\|_{L^2_\xi}\\
 &\lesssim \int_{-1}^1  \frac{1}{(1-\mu^2)^{1/4}} |\hat{h}(\mu)| \d \mu \lesssim \| \chi_{[-1,1]} \hat{f}\|_{L^2}.
\end{align*}

For $B$, recall that $B =L^t \phi(x)$, where $L^t$ is the Fourier multiplier operator given by $e^{-it \mu  \sqrt{1 + 1/\mu^2}}$ and $\phi$ is given above in the proof of Lemma \ref{ibvp_xsb}. Then the fact that $B \in C_t^0H_x^s$ follows from the time continuity of $L^t$ and the $H^s$ bounds on $\phi$ derived previously.

To bound the solution in $H^s_t$, note that 
\begin{align*}
\| A \|_{H^s_t} &= \| \lb \mu \rb^s \chi_{[-1,1]}(\mu) f(x \sqrt{1-\mu^2}) \hat{h}(\mu) \|_{L^2_\mu} \lesssim \| \lb \mu \rb^s \chi_{[-1,1]}(\mu) \hat{h}(\mu) \|_{L^2_\mu}, \\
 \| B\|_{H^s_t} &= \| \lb \mu \rb^s \chi_{|\mu| \geq 1}(\mu) e^{-ix\mu\sqrt{1-1/\mu^2}} \hat{h}(\mu) \|_{L^2_\mu} = \| \lb \mu \rb^s \chi_{|\mu| \geq 1}(\mu)\hat{h}(\mu) \|_{L^2_\mu}.
\end{align*}

\subsection{Proof of Lemma \ref{wave_duhamel_est}: Kato Smoothing for the Klein-Gordon Duhamel Term} \label{wave_duhamel_est_prf}
We consider the `$+$' case; the `$-$' case can be treated in the same way.
Again, it suffices to prove the bound for $x=0$. We have
\begin{align*}
 \eta(t)\int_0^t e^{ (t- t^\prime)D} G  \d t^\prime \Big|_{x=0} 
 &=   \eta(t)  \iint \frac{e^{it \lambda }-e^{ it  \operatorname{sgn}(\xi) \lb\xi \rb }}{i(\lambda- \operatorname{sgn}(\xi) \lb\xi \rb )} \psi(\lambda-  \operatorname{sgn}(\xi) \lb\xi \rb ) \widehat G(\xi,\lambda) \d\xi \d\lambda \\ 
 &+\eta(t) \iint \frac{e^{it \lambda } }{i(\lambda- \operatorname{sgn}(\xi) \lb\xi \rb )} \psi^c(\lambda-  \operatorname{sgn}(\xi) \lb\xi \rb ) \widehat G(\xi,\lambda) \d\xi \d\lambda \\
&-\eta(t) \iint \frac{ e^{ it  \operatorname{sgn}(\xi) \lb\xi \rb }}{i(\lambda-  \operatorname{sgn}(\xi) \lb\xi \rb )} \psi^c(\lambda-  \operatorname{sgn}(\xi) \lb\xi \rb ) \widehat G(\xi,\lambda) \d\xi \d\lambda \\ &=:\I+\II-\III.
\end{align*}

The $\I$ term can be bounded by $\|G\|_{Y^{s,-b}_+}$ using a Taylor expansion argument just as in \cite[Proposition 3.8]{ET1}. To bound $\III$, we calculate 
\begin{align*}
 \| \III\|_{H^s} &= \left\|\lb \tau\rb^s \iint  \frac{\hat{\eta}(\tau - \operatorname{sgn}(\xi)\lb \xi \rb) }{i(\lambda-  \operatorname{sgn}(\xi) \lb\xi \rb )} \psi^c(\lambda-  \operatorname{sgn}(\xi) \lb\xi \rb ) \widehat{G}(\xi,\lambda) \d\xi \d\lambda \right\|_{L^2_\tau} \\
 &\lesssim \left\| \lb \tau \rb^s \iint  \frac{|\hat{\eta}(\tau - \sgn(\xi)\lb \xi \rb)| }{\lb \lambda-  \xi \rb} |\widehat{G}(\xi,\lambda)| \d\xi \d\lambda \right\|_{L^2_\tau}.
\end{align*}
Using the fact that $\lb \tau \rb^s \lesssim \lb \xi \rb^s \lb \tau - \xi \rb^{|s|}$ and $\hat{\eta}(\tau - \sgn(\xi) \lb \xi \rb) \lesssim \lb \tau - \xi \rb^{-|s|-2}$ and then the Cauchy-Schwarz inequality with $b < \frac12$, the above is bounded by
\begin{align*}
\left\|\iint  \frac{1}{\lb \lambda-  \xi \rb \lb \tau - \xi \rb^2}  \lb \xi \rb^s |\widehat{G}(\xi,\lambda)| \d\xi \d\lambda \right\|_{L^2_\tau} 
&\lesssim \left\|\int  \frac{1}{ \lb \tau - \xi \rb^2}  \left\|  \lb \xi \rb^s \lb \lambda-  \xi \rb ^{-b} \widehat{G}(\xi,\lambda)\right\|_{L^2_\lambda} \d\xi \right\|_{L^2_\tau}.
\end{align*}
By Young's inequality, this is bounded by $\left\|  \lb \xi \rb^s \lb \lambda-  \xi \rb ^{-b} \widehat{G}(\xi,\lambda)\right\|_{L^2_\lambda L^2_\tau} = \|G\|_{Y^{s,-b}_+}$. Next, we have 
\begin{equation*}
 \| \II\|_{H^s_t} \lesssim \left\| \lb \lambda \rb^s \int \frac{1}{\lb \lambda- \xi\rb} |\widehat G(\xi,\lambda)| \d\xi \right\|_{L^2_\lambda}.
\end{equation*}
For $s < 0$, this is bounded by 
\begin{equation*}
 \left\| \lb \lambda \rb^s \int_{|\xi| \gg |\lambda|} \frac{1}{\lb \lambda- \xi\rb} |\widehat G(\xi,\lambda)| \d\xi \right\|_{L^2_\lambda} +
\left\| \int \lb \xi \rb^s \frac{1}{\lb \lambda- \xi\rb} |\widehat G(\xi,\lambda)| \d\xi \right\|_{L^2_\lambda}.
\end{equation*}
The second norm in the previous line is bounded by $\| G\|_{Y^{s,-b}_+}$, so we're done with the $s< 0$ case. The remaining cases are the same as those treated in \cite[Proposition 3.8]{ET1} and are omitted.

\subsection{Proof of Lemma \ref{prop:smooth23}}\label{prop:smooth23_prf}

Using the convolution structure of the Fourier transform, write
\begin{equation*}
 \mathcal{F}(D^{-1}(u\overline{v})(\xi ,\lambda) = \lb\xi\rb^{-1} \iint \hat{u}(\xi +\xi_1, \lambda + \lambda_1) \overline{\hat{v}}(\xi_1, \lambda_1) \d \xi_1 \d \lambda_1.
\end{equation*}
Let $f(\xi, \lambda) = |\hat{u}(\xi,\lambda)|$ and $g(\xi, \lambda) =  |\hat{v}(\xi,\lambda)|$. Since on the domain of integration $ \lambda \mp \xi  \approx \xi$, it suffices to show that 
\begin{multline} \label{eq:eq1}  
\left\| \lb \lambda \rb ^{s_1 + a_1} \iiint_{|\xi| \gg |\lambda|} f(\xi + \xi_1, \lambda + \lambda_1) g(\xi_1, \lambda_1) \d \xi_1 \d \lambda_1 \d \xi \right\|_{L^2_\lambda} \\
\lesssim \|\lb \xi \rb^{s_0} \lb \lambda - \xi^2 \rb^b f \|_{L^2L^2}\|\lb \xi \rb^{s_0} \lb \lambda-\xi^2 \rb^b g\|_{L^2L^2}. 
\end{multline}

Note also that $\lb \lambda \rb^{s_1+a_1}/\lb \xi \rb^2 \ll \lb \lambda \rb^{s_1 + a_1 - 2}$. The left-hand side of \eqref{eq:eq1} is bounded by 
\begin{align*}
 \left\| \lb \lambda \rb^{s_1 + a_1 -2} \int_{|\xi| \gg |\lambda|} [f \ast_{\xi,\lambda} g(-\cdot, -,\cdot)](\xi, \lambda) \d \xi \right\|_{L^2_\lambda}.
\end{align*}
Using Young's inequality to bound the $L^1_\xi$ norm of the convolution (noting that the functions $f$ and $g$ are nonnegative) and then the Cauchy-Schwarz inequality, we arrive at the bounds
\begin{align*}
 \left\| \lb \lambda \rb^{s_1 + a_1 -2} \|f\|_{L^1_\xi}  \ast_{\lambda} \|g(-\cdot,-\cdot)\|_{L^1_\xi} \right\|_{L^2_\lambda} \lesssim \left\|  \|{f}\|_{L^1_\xi}  \ast_{\lambda} \|{g}(-\cdot,-\cdot)\|_{L^1_\xi} \right\|_{L^\infty_\lambda}.
\end{align*}
Using Young's inequality again, this is bounded by $\|{f}\|_{L^2_\lambda L^1_\xi}\|{g}\|_{L^2_\lambda L^1_\xi}$. Now 
\[ \|{f}\|_{L^2_\lambda L^1_\xi} = \left\| \int \frac{ \lb \xi \rb^{s_0} \lb \lambda - \xi^2 \rb^{b} |\hat{u}|(\xi, \lambda)}{\lb \xi \rb^{s_0} \lb \lambda - \xi^2 \rb^{b}} \d \xi \right\|_{L^2_\lambda} \lesssim \left( \sup_\lambda \int \lb \xi \rb^{-2s_0} \lb \lambda - \xi^2 \rb^{-2b} \d \xi \right)^{1/2} \|u\|_{X^{s_0,b}} .\]
For $s_0 > -\frac12$, change variables in the supremum by setting $\rho = \xi^2$ to see that the supremum is finite as long as $s_0 + 2b > \frac12$. The same procedure bounds $\|{g}\|_{L^2_\lambda L^1_\xi}$ in terms of $\|v\|_{X^{s_0,b}}$, so we're done.

\subsection{Proof of Lemma \ref{T_power_est}} \label{T_power_est_prf} 

Combining the first part of Lemma \ref{schr_ibvp_xsb}, Lemma \ref{char_func_lemma}, and Proposition \ref{schr_duhamel_est}, we have 
\begin{equation}\label{eq:unscaled_result} \| \eta(t) W^t_0(0,q) \|_{X^{0,b}} \lesssim \| F\|_{X^{0,-b}}, 
\end{equation}

where 
\[ q(t) = \Bigl[ \eta(t) \int_0^t e^{i(t-t') \Delta} F(t') \d t' \Bigr]_{x=0}. \] 

A calculation (for details, see \cite{ET}) shows that $q$ can be written in the form 
\[ q(t) = \eta(t) \iint \frac{e^{it\lambda} - e^{-i t \xi^2}}{i(\lambda + \xi^2)} \widehat{F}(\xi,\lambda) \d \xi \d \lambda.\]

The formula for $W^t_0(0,q)$ uses the Fourier transform of $\chi q$, which is 
\[ \widehat{\chi q}(\tau) = \iint \frac{\hat{\chi \eta}(\tau - \lambda) - \hat{\chi \eta}(\tau + \xi^2)}{i(\lambda + \xi^2)} \widehat{F}(\xi,\lambda) \d \xi \d \lambda.\]

The linear flow $W_0^t$ can be written in the form (\cite{bonaetal}, \cite{ET})
\begin{multline*} W^t_0(0,q) = \frac1\pi \int \chi(\beta) e^{-i \beta^2 t + i \beta x } \beta \widehat{\chi q}(-\beta^2) \d \beta + \frac1\pi \int \chi(\beta) e^{i \beta^2 t - \beta x} \rho( \beta x) \beta \hat{\chi q}(\beta^2) \d \beta\\ =: \frac1\pi(A + B). \end{multline*}

Then, using the fact that $A$ is an inverse spatial Fourier transform, we have
\begin{align*}
 \hat{\eta A}(\beta, \omega) = \chi(\beta) \beta \iint  \hat{\eta}(\omega + \beta^2)  \frac{\hat{\chi \eta}(-\beta^2 - \lambda) - \hat{\chi \eta}(-\beta^2 + \xi^2)}{i(\lambda + \xi^2)} \widehat{F}(\xi,\lambda) \d \xi \d \lambda.
\end{align*}

The estimate \eqref{eq:unscaled_result} implies that 
\[ \| \lb \omega + \beta^2 \rb ^b \hat{\eta A}(\beta, \omega) \|_{L^2_{\beta, \omega}} \lesssim \| F\|_{X^{0,-b}}.\]
[Technically, \eqref{eq:unscaled_result} gives a bound on $A + B$, but $A$ and $B$ are treated separately in the proof, so the bound holds for each.] 

Writing this out explicitly and setting $G(\xi, \lambda) = \lb \lambda + \chi^2 \rb ^{-b} \hat{F}(\xi,\lambda)$, we get
\begin{multline} \label{eq:expl_unscaled_result} \left\| \lb \omega + \beta^2 \rb ^b  \chi(\beta) \beta \iint  \hat{\eta}(\omega + \beta^2)  \frac{\hat{\chi \eta}(-\beta^2 - \lambda) - \hat{\chi \eta}(-\beta^2 + \xi^2)}{i(\lambda + \xi^2)\lb \lambda + \xi^2 \rb^{-b}}  G(\xi,\lambda) \d \xi \d \lambda \right\|_{L^2_{\beta, \omega}} \\ \lesssim \|G\|_{L^2_{\xi,\lambda}}. 
\end{multline}

Using the fact that $\hat{\eta_T}(\tau) = T \hat{\eta}(T \, \tau)$, the quantity which we need to bound can be written as 
\begin{multline*}  
T^2 \; \times \\
\left\| \lb \omega + \beta^2 \rb ^b  \chi(\beta) \beta \iint  \hspace{-6pt}\hat{\eta}(T(\omega + \beta^2))  \frac{\hat{\chi \eta}(T(-\beta^2 - \lambda)) - \hat{\chi \eta}(T(-\beta^2 + \xi^2))}{i(\lambda + \xi^2)\lb \lambda + \xi^2 \rb^{-b}}  G(\xi,\lambda) \d \xi \d \lambda \right\|_{L^2_{\beta, \omega}}.
\end{multline*}
Now rescale all four variables by letting
\[ \sqrt{T} \beta \mapsto \beta \qquad T \omega \mapsto \omega \qquad \sqrt{T} \xi \mapsto \xi \qquad T \lambda \mapsto \lambda. \]
After this rescaling, the quantity above becomes
\begin{multline*} T^{\frac14}  \; \times \\ \left\| \lb (\omega + \beta^2)/T \rb ^b  \chi(\beta) \beta \iint \hspace{-6pt} \hat{\eta}(\omega + \beta^2)  \frac{\hat{\chi \eta}(-\beta^2 - \lambda) - \hat{\chi \eta}(-\beta^2 + \xi^2)}{i(\lambda + \xi^2)\lb (\lambda + \xi^2)/T \rb^{-b}}  G\left(\frac{\xi}{\sqrt{T}},\frac{\lambda}{T}\right) \d \xi \d \lambda \right\|_{L^2_{\beta, \omega}} \hspace{-8pt}.
\end{multline*}

If $T \approx 1$, there is nothing to prove, so we may assume $T \ll 1$. Then $\lb x/T \rb \lesssim \lb x \rb/T$. Since $b$ is positive, the norm above can be bounded by 
\begin{multline*}
 T^{\frac14 - 2b}  \; \times \\ \left\| \lb \omega + \beta^2 \rb ^b  \chi(\beta) \beta \iint  \hat{\eta}(\omega + \beta^2)  \frac{\hat{\chi \eta}(-\beta^2 - \lambda) - \hat{\chi \eta}(-\beta^2 + \xi^2)}{i(\lambda + \xi^2)\lb \lambda + \xi^2 \rb^{-b}}  G\left(\frac{\xi}{\sqrt{T}},\frac{\lambda}{T}\right) \d \xi \d \lambda \right\|_{L^2_{\beta, \omega}}.
\end{multline*}
The result \eqref{eq:expl_unscaled_result} bounds this by
\[ T^{\frac14 - 2b} \| G(\xi/\sqrt{T},\lambda/T) \|_{L^2_{\xi,\lambda}} = T^{1-2b} \|G\|_{L^2_{\xi,\lambda}} = T^{1-2b} \| F\|_{X^{0,-b}},\]
as desired. 

The $B$ term can be treated similarly. We have 
\begin{multline*} \hat{\eta B}(\mu, \omega) = \\ \iiint \hat{\eta}(\omega - \beta^2) \mathcal{F}_x\left( e^{-\beta x} \rho(\beta x) \right)(\mu) \beta \chi(\beta) \frac{\hat{\chi \eta}(-\beta^2 - \lambda) - \hat{\chi \eta}(-\beta^2 + \xi^2)}{i(\lambda + \xi^2)\lb \lambda + \xi^2 \rb^{-b}}  G(\xi,\lambda) \d \beta \d \xi \d \lambda . 
\end{multline*}
Letting $f(y) = e^{-y} \rho(y)$, this is equal to
\begin{equation*} \iiint \hat{\eta}(\omega - \beta^2) \hat{f}(\mu/\beta) \chi(\beta) \frac{\hat{\chi \eta}(-\beta^2 - \lambda) - \hat{\chi \eta}(-\beta^2 + \xi^2)}{i(\lambda + \xi^2)\lb \lambda + \xi^2 \rb^{-b}}  G(\xi,\lambda) \d \beta \d \xi \d \lambda,
\end{equation*}
and we know from \eqref{eq:unscaled_result} that
\begin{multline*}
 \left\| \lb \omega + \mu^2 \rb^b \iiint \hat{\eta}(\omega - \beta^2) \hat{f}(\mu/\beta) \chi(\beta) \frac{\hat{\chi \eta}(-\beta^2 - \lambda) - \hat{\chi \eta}(-\beta^2 + \xi^2)}{i(\lambda + \xi^2)\lb \lambda + \xi^2 \rb^{-b}}  G(\xi,\lambda) \d \beta \d \xi \d \lambda \right\|_{L^2_{\mu,\omega}} \\
 \lesssim \|G\|_{L^2_{\xi,\lambda}}. 
\end{multline*}
The desired bound for $\eta_T B$ is obtained from this just as above by scaling all five variables.

\appendix
\section{Proof of Klein-Gordon Solution Formula on $\Bbb R^+$} \label{appendA}

Taking the Laplace transform in time of \eqref{eq:KGibvp} yields the equation
\[ \begin{cases}
    \lambda^2 \wt{n} - \Delta \wt{n} + \wt{n} = 0, \\
    \wt{n}(0,\lambda) = \wt{h}(\lambda).
   \end{cases} \]
The characteristic equation is $\lambda^2 - w^2 + 1 = 0$, which has roots $w = \pm \sqrt{\lambda^2 + 1}$. Since we are concerned with solutions that decay at infinity, we take the negative root. Thus we have 
\[ \wt{n}(x,\lambda) = e^{-x\sqrt{\lambda^2 + 1}}\, \wt{h}(\lambda).\]
Note that $\sqrt{\lambda^2 + 1}$ can be defined analytically on $\mathbb{C}\backslash [-i,i]$. By Mellin inversion, we have, for any $c > 0$, the equality 
\begin{equation*}
 n(x,t) = \frac1{2\pi i} \int_{c - i \infty} ^{c + i \infty} e^{\lambda t - x \sqrt{\lambda^2 + 1}} \, \wt{h}(\lambda) \d \lambda = \Re \frac1{\pi i} \int_{c +0 i } ^{c + i \infty} e^{\lambda t - x \sqrt{\lambda^2 + 1}} \, \wt{h}(\lambda) \d \lambda.
\end{equation*}
Taking $c \to 0^+$, we arrive at 
\begin{align*}
 n(x,t) &= \Re \frac1{\pi i} \int_{0} ^{i \infty} e^{\lambda t - x \sqrt{\lambda^2 + 1}} \, \wt{h}(\lambda) \d \lambda\\
 &= \Re \frac1{\pi} \int_{0} ^\infty e^{ i \mu t - x \sqrt{1 - \mu^2}} \, \widehat{h}(\mu) \d \mu.\\
\end{align*}
To ensure convergence when $x <0$, we include a smooth cut-off function $\rho$, as follows:
\begin{align*}
 n(x,t) &= \Re \frac1{\pi} \int_{0} ^1 e^{ i \mu t - x \sqrt{1 - \mu^2}} \rho\bigl(x\sqrt{1 - \mu^2}\bigr) \hat{h}(\mu) \d\mu + \Re \frac1{\pi} \int_{1} ^\infty e^{ i \mu t - x \sqrt{1 - \mu^2}} \hat{h}(\mu) \d \mu.
\end{align*}
Changing variables, this can be written as $\frac1{2\pi} (A + B)$, where
\begin{align*}
 A &=  \int_{-1} ^1 e^{ i \mu t - x \sqrt{1 - \mu^2}} \rho(x\sqrt{1 - \mu^2}) \hat{h}(\mu) \d\mu  \\
 B &= \int_{-\infty} ^\infty e^{-it \mu \sqrt{1 + 1/\mu^2} + i\mu x } \hat{h}(-\mu \sqrt{1+ 1/\mu^2}) \frac{1}{\sqrt{1 + 1/\mu^2}} \d \mu. 
\end{align*}
The rather cumbersome variables in $B$ are necessary to ensure we arrive at the correct branch of the square root function.

\end{document}